\documentclass{siamart0216}

\usepackage{amsmath}
\usepackage{amsmath, hyperref, color}
\usepackage{graphicx}
\usepackage{caption}
\usepackage{subcaption}

\usepackage[shortlabels]{enumitem}
\usepackage{graphicx}
\usepackage[all]{xypic}
\usepackage{makecell}

\usepackage{lipsum}

\usepackage{lipsum}
\usepackage{amsfonts}
\usepackage{graphicx}
\usepackage{epstopdf}
\usepackage{algorithmic}
\ifpdf
  \DeclareGraphicsExtensions{.eps,.pdf,.png,.jpg}
\else
  \DeclareGraphicsExtensions{.eps}
\fi

\newcommand{\TheTitle}{Convexity in Tree Spaces} 
\newcommand{\TheAuthors}{Bo Lin, Bernd Sturmfels, Xiaoxian Tang and Ruriko Yoshida}

\headers{\TheTitle}{\TheAuthors}

\title{{\TheTitle}}

\author{
 Bo Lin\thanks{Department of Mathematics, University of California, Berkeley
    (\email{linbo@berkeley.edu}).}
  \and
  Bernd Sturmfels\thanks{Department of Mathematics, University of
    California, Berkeley (\email{bernd@berkeley.edu}).}
\and
Xiaoxian Tang\thanks{Department of Mathematics, Universit\"at Bremen (\email{xtang@uni-bremen.de}).}
  \and
  Ruriko Yoshida\thanks{Department of Statistics, 
    University of Kentucky, (\email{ruriko.yoshida@uky.edu}).}
}

\usepackage{amsopn}

\ifpdf
\hypersetup{
  pdftitle={\TheTitle},
  pdfauthor={\TheAuthors}
}
\fi

\newtheorem{problem}[theorem]{Problem}

\newtheorem{example}[theorem]{Example}

\newtheorem{algorithm2}[theorem]{Algorithm}
\newtheorem{experiment}[theorem]{Experiment}

\newcommand{\RR}{\mathbb{R}}

\newcommand{\para}{{\Vert}}
\DeclareMathOperator*{\argmin}{arg\,min}

\begin{document}

\maketitle

\begin{abstract}
We study the geometry of metrics and convexity structures on the space of 
phylogenetic trees, which is here realized as the tropical linear space of all ultrametrics.
The ${\rm CAT}(0)$-metric of Billera-Holmes-Vogtman arises
from the theory of orthant spaces. 
While its geodesics can be computed by
the Owen-Provan algorithm, geodesic triangles are complicated.
We show that the dimension of such a triangle can be arbitrarily high.
Tropical convexity and the tropical metric behave better.
They exhibit properties  desirable for geometric statistics,
such as geodesics of small depth.
\end{abstract}

\begin{keywords}
 Billera-Holmes-Vogtman metric, ${\rm CAT}(0)$ metric space, 
 geodesic triangle,  phylogenetic tree, polytope,  tropical convexity.
\end{keywords}

\begin{AMS}
05C05, 30F45,  52B40,  68U05, 92D15
\end{AMS}

\section{Introduction}
\label{sec1}

A finite metric space with $m$ elements is represented by a 
nonnegative symmetric $m \times m$-matrix $D = (d_{ij})$
with zero entries on the diagonal such that all 
   triangle inequalities are satisfied:
$$ d_{ik} \leq d_{ij} + d_{jk} 
\quad \hbox{for all} \,\,\, i,j,k \,\,\, \hbox{in} \,\,\, [m] :=  \{1,2,\ldots,m\} . $$
The set of all such metrics is a full-dimensional
closed polyhedral cone, known as the {\em metric cone},
 in the space $\RR^{\binom{m}{2}}$ of symmetric matrices with vanishing diagonal.
For many applications one requires the following
strengthening of the triangle inequalities:
 \begin{equation}
 \label{eq:ultrametric} \quad
d_{ik} \leq  {\rm max}(d_{ij} , d_{jk} )
\quad \hbox{for all} \,\,\, i,j,k \in [m]. 
\end{equation}
If \eqref{eq:ultrametric} holds then the metric space $D$ is called an {\em ultrametric}.
The set of all  ultrametrics contains the
ray $ \mathbb{R}_{\geq 0} {\bf 1}$ spanned by the all-one metric ${\bf 1}$,
which is defined by $d_{ij} = 1$ for $1 \leq i < j \leq m$.
The image of the set of ultrametrics in the quotient 
space $\mathbb{R}^{\binom{m}{2}} \!/ \mathbb{R} {\bf 1}$
is denoted $\mathcal{U}_m$ and called the {\em space of ultrametrics}.
It is known in tropical  geometry \cite{AK, MS}  and in
phylogenetics \cite{GD,SS} that  $\mathcal{U}_m$ 
is the support of a pointed simplicial fan of dimension $m-2$.
That fan has $2^m{-}m{-}2$ rays, namely the
{\em clade metrics} $D_\sigma$. A {\em clade} $\sigma$ is
a proper subset of $[m]$ with at least two elements,
and $D_\sigma$ is the ultrametric whose $ij$-th entry 
is $0$ if $i,j \in \sigma$ and $1$ otherwise.
Each cone in that fan structure consists of all ultrametrics
whose tree has a fixed topology. 
We encode each topology by a {\em nested set} \cite{Fei}, i.e.~a 
set of clades $\{\sigma_1,\sigma_2,\ldots,\sigma_d\}$  such that
\begin{equation}
\label{eq:nestedset}
\sigma_i \subset \sigma_j \quad \hbox{or} \quad
\sigma_j \subset \sigma_i \quad \hbox{or} \quad
\sigma_i \,\cap \,\sigma_j = \emptyset
\qquad \hbox{for all} \,\, \,1 \leq i < j \leq d .
\end{equation}
Here $d$ can be any integer between $1$ and $m-2$.
The nested set represents the $d$-dimensional cone
spanned by $\{D_{\sigma_1}, D_{\sigma_2},\ldots,D_{\sigma_d}\}\,$
inside $\,\mathcal{U}_m \subset \mathbb{R}^{\binom{m}{2}} \! / \mathbb{R} {\bf 1}$.
For an illustration of this fan structure, consider 
equidistant trees on $m=4$ taxa. The space of these is a
two-dimensional fan over the {\em Petersen graph}, shown on the left in Fig.~\ref{fig:eins}.

\begin{figure}[h]
  \begin{center}
    \includegraphics[width=5.5cm]{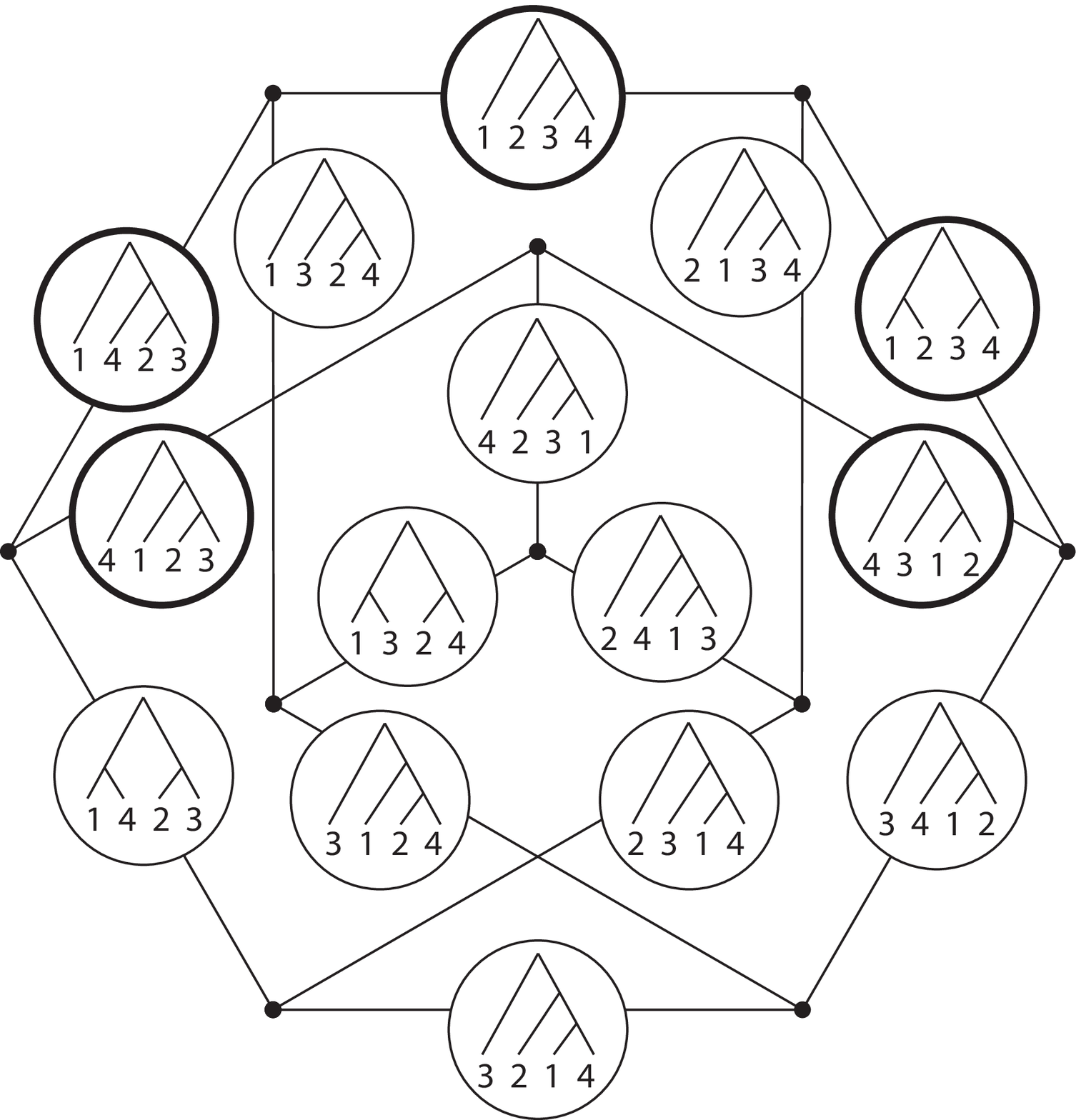}  \qquad
        \includegraphics[width=5.5cm]{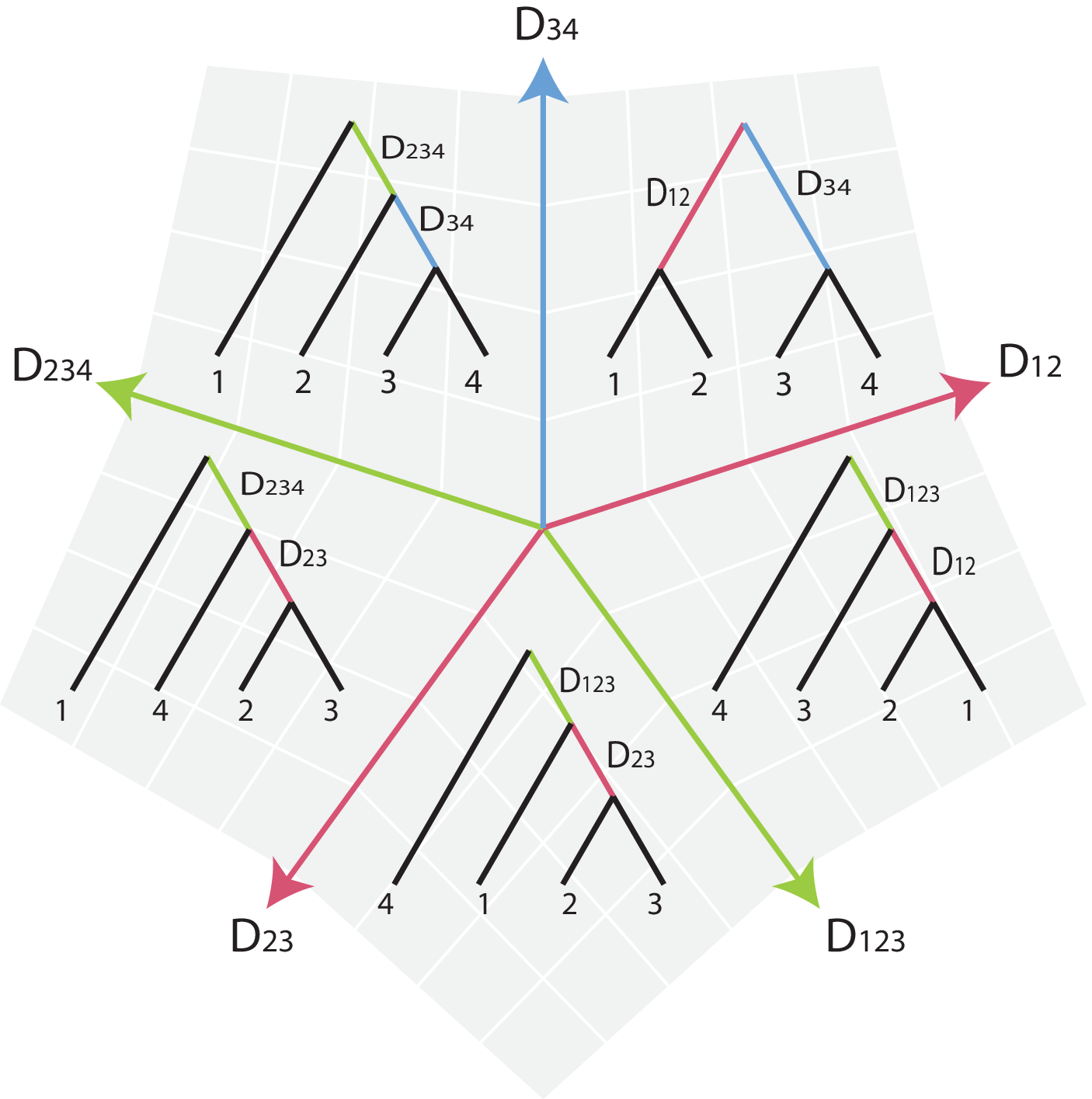} \qquad \qquad
  \end{center}
  \caption{\label{fig:eins}   The space of ultrametrics $\mathcal{U}_4$ 
  is a two-dimensional fan with  $15$ maximal cones.
  Their adjacency forms a Petersen graph.
  Depicted on the right is a cycle of five  cones.}
  \end{figure}
  
At this point it is essential to stress that we have {\bf not} yet defined 
convexity or a metric on $\mathcal{U}_m$.
 So far, our tree space $\mathcal{U}_m$ is nothing
 but a subset of $\mathbb{R}^{\binom{m}{2}} \! / \mathbb{R} {\bf 1}$.
 It is the support of the fan described above,  but even that fan structure is not unique.
There are other meaningful fan structures, classified by the
{\em building sets} in \cite{Fei}. An important one
  is the $\tau$-space of \cite[\S 2]{GD},
known to combinatorialists as the order complex of the partition lattice~\cite{AK}.

The aim of this paper is to compare different
 geometric structures on $\mathcal{U}_m$,
 and to explore  statistical issues.
The first geometric  structure is the metric proposed by
Billera, Holmes and Vogtman in \cite{BHV}. In their setting,
each cone is right-angled. The {\em BHV metric}
is the unique metric on $\mathcal{U}_m$ that restricts to the usual Euclidean
distance on each such right-angled cone. For this to be well-defined,
we must fix a simplicial fan structure on $\mathcal{U}_m$.
This issue is subtle, as  explained by Gavruskin and Drummond in  \cite{GD}.
The BHV metric has the ${\rm  CAT}(0)$-property \cite[\S 4.2]{BHV}.
This implies that between any two points there is a unique geodesic.

Owen and Provan \cite{OP} proved that these geodesics can be computed in polynomial time.
In Section~\ref{sec2}, we present a detailed review and analysis.
This is done in the setting of orthant spaces $\mathcal{F}_\Omega$
associated with flag simplicial complexes $\Omega$.

In Section~\ref{sec3} we study geodesically closed
subsets of an orthant space $\mathcal{F}_\Omega$,
with primary focus on geodesic triangles.
Problem~\ref{prob:berger} asks whether 
these are always closed. Our main result, 
Theorem \ref{thm:highdim}, states that 
the dimension of a geodesic triangle can  be arbitrarily large.
The same is concluded for the tree space $\mathcal{U}_m$ in Corollary~\ref{cor:treebig}.
For experts in phylogenetics,  we note that our results are not restricted to equidistant trees.
They extend naturally to the more familiar BHV space for all rooted trees.

Tropical geometry \cite{MS} furnishes an
alternative geometric structure on $\mathcal{U}_m$,
via  the graphic matroid of the complete graph \cite[Example 4.2.14]{MS}.
More generally, for any matroid $M$, the tropical linear space $ {\rm Trop}(M)$
 is tropically convex by \cite[Proposition 5.2.8]{MS},
and it is a metric space with the tropical distance to be defined in (\ref{eq:tropmetric}).
Tropical geodesics are not unique, but tropically convex sets
have desirable properties. In particular, triangles are always $2$-dimensional.

Section~\ref{sec5} offers an experimental study of Euclidean
geodesics and tropical segments. The
latter are better than the former with regard to {\em depth},
i.e.~the largest codimension of  cones traversed. This
is motivated by the issue of {\em stickiness} in geometric statistics \cite{openbook, Mil}.
Section~\ref{sec6} advocates tropical geometry for statistical applications.
Starting from Nye's principal component analysis in \cite{Nye},
we propose two basic tools for future data analyses:
computation of tropical centroids,
and nearest-point projection onto tropical linear spaces.

In a statistical context, it can be advantageous
to  replace $\,\mathcal{U}_m$ by a compact subspace.
We define  {\em compact tree space}   $\,\mathcal{U}_m^{[1]}\,$
to be the image in $\mathcal{U}_m$ of the set of ultrametrics $D = (d_{ij})$
that satisfy ${\rm max}_{ij} \{d_{ij}\} = 1$.
This is a polyhedral complex consisting of one convex polytope
for each nested set $\{\sigma_1,\sigma_2,\ldots,\sigma_d\}$.
In the notation of (\ref{eq:Drep}), this polytope
consists of all $(\ell_1,\ell_2,\ldots,\ell_d) \in [0,1]^d$
such that  $\ell_{i_1} {+} \ell_{i_2}  {+} \cdots  {+} \ell_{i_d} \leq 1$
 whenever $\sigma_{i_1} \subset \sigma_{i_2} \subset                            
 \cdots \subset \sigma_{i_d}$.  In  phylogenetics,  these are 
equidistant trees of height $\frac{1}{2}$ with a fixed tree topology.
For instance,  $\mathcal{U}_4^{[1]}$ is a polyhedral surface, consisting of
$12$ triangles and $3$ squares, glued along $10$ edges.

\section{Orthant Spaces}
\label{sec2}

In order to understand the geometry of tree spaces, we
 work in the more general setting
of globally nonpositively curved (NPC) spaces. This was suggested by
Miller, Owen and Provan in \cite[\S 6]{MOP}.
We follow their set-up.

Consider a simplicial complex $\Omega$ on
the ground set $[n] = \{1,2,\ldots,n\}$. We say that
$\Omega$ is a {\em flag complex} if all
its minimal non-faces have two elements. Equivalently,
a flag  complex is determined by its edges:
a subset $\sigma$ of $ [n]$ is in $\Omega$ if
and only if $\{i,j\} \in \Omega$ for all $i,j \in \sigma$.
Every simplicial complex $\Omega$ on $[n]$
determines a  simplicial fan $\mathcal{F}_\Omega$ in $\RR^n$.
The cones in $\mathcal{F}_\Omega$ are the orthants
$\mathcal{O}_\sigma = {\rm pos} \{ e_i: i \in \sigma \} $ where $\sigma \in \Omega$.
Here $\{e_1,e_2,\ldots,e_n\}$ is the standard basis of $\RR^n$.
We say that $\mathcal{F}_\Omega$ is an {\em  orthant space}
if the underlying simplicial complex $\Omega$ is flag.
The support of the fan $\mathcal{F}_\Omega$ is turned into a metric space by
fixing the usual Euclidean distance on each orthant.
A path of minimal length between two points is called
a {\em geodesic}.

\begin{proposition}
\label{prop:geodesics}
Let $\mathcal{F}_\Omega$ be an orthant space. For any two points $v$ and $w$ in 
$\mathcal{F}_\Omega$ there exists a unique geodesic between $v$ and $w$.
This geodesic is denoted $G(v,w)$.
\end{proposition}

The uniqueness of geodesics is attributed to Gromov. We refer to
\cite[Theorem 2.12]{ABY} and \cite[Lemma 6.2]{MOP} for expositions
and applications of this important result.
The main point is that orthant spaces, with their Euclidean metric as above,
satisfy the  ${\rm CAT}(0)$ property, provided $\Omega$ is flag. This property states that
chords in triangles are no longer than the corresponding chords
in Euclidean triangles. The metric spaces $\mathcal{F}_\Omega$
coming from flag complexes $\Omega $ 
 are called {\em global NPC orthant spaces} 
in \cite{MOP}.  For simplicity we here use the term {\em orthant space}
for $\mathcal{F}_\Omega$.

\begin{example} \rm
Let $n=3$ and $\Omega = \{12, 23, 31\}$,  i.e.~the $3$-cycle. The fan  $\mathcal{F}_\Omega$ is the
boundary of the nonnegative orthant in $\RR^3$. This is not an orthant
space because $\Omega$ is not flag.
Some geodesics in $\mathcal{F}_\Omega$ are not unique:
 the points $v = (1,0,0)$ and $w = (0,1,1)$
have distance $\sqrt{5}$, and there are two geodesics:
one passing through  $(0,\frac{1}{2},0)$ and the other passing through $(0,0,\frac{1}{2})$.
By contrast, let $n=4$ and $\Omega = \{12, 23, 34, 41 \}$, i.e.~the $4$-cycle.
Then  $\mathcal{F}_\Omega$  is a $2$-dimensional orthant space in $\RR^4$.
The Euclidean geodesics on that surface are unique.
\end{example}

The problem of computing the unique geodesics
was solved by Owen and Provan in \cite{OP}.
In \cite{OP}, the focus was
on tree space of Billera-Holmes-Vogtman \cite{BHV}.
It was argued in \cite[Corollary 6.19]{MOP} that  the result extends
to arbitrary orthant spaces.
Owen and Provan gave a polynomial-time algorithm whose input
consists of two points $v$ and $w$ in $\mathcal{F}_\Omega$
and whose output is the geodesic $G(v,w)$.
  We shall now describe their method.
    
Let $\sigma$ be a simplex in a flag complex $\Omega$, with corresponding
orthant $ \mathcal{O}_\sigma $ in $\mathcal{F}_\Omega$. Consider a point 
$v = \sum_{i \in \sigma} v_i e_i $ in $\mathcal{O}_\sigma$
and any face $\tau$ of $\sigma$. We write
$v_\tau = \sum_{i \in \tau} v_i e_i$ for the projection
of $v$ into $\mathcal{O}_\tau$. Its Euclidean length
$|| v_\tau || = \bigl( \sum_{i \in \tau} v_i^2 \bigr)^{1/2}$ is called
the {\em projection length}.

We now assume that $\Omega$ is pure $(d-1)$-dimensional,
i.e.~all maximal simplices in $\Omega$ have the same
dimension $d-1$. This means that all maximal orthants
in $\mathcal{F}_\Omega$ have dimension $d$.
Consider two general points $v$ and $w$ 
in the interiors of full-dimensional orthants
$ \mathcal{O}_\sigma$ and $\mathcal{O}_\tau$ 
of $\mathcal{F}_\Omega$ respectively.
We also assume that $\sigma \cap \tau = \emptyset$.
 Combinatorially, the geodesic $G(v,w)$ is
then encoded by a pair $({\cal A},{\cal B})$ where
${\cal A}=(A_1,\ldots,A_q)$ is an ordered partition
of $\sigma = \{\sigma_1,\ldots,\sigma_d\}$ 
 and ${\cal  B}=(B_1,\ldots,B_q)$ 
is an ordered partition of $\tau = \{\tau_1,\ldots, \tau_d \}$.
These two partitions have the same number $q$ of parts, and they
 satisfy the following three properties:
\begin{itemize}
\item[]{\bf (P1)} for all pairs $i>j$,  the set $A_i \cup B_j$
is a simplex in $\Omega$.
\item[]{\bf (P2)}  $  
||v_{A_1}||/||w_{B_1}||  \,\leq \,
||v_{A_2}||/||w_{B_2}||  \,\leq \,\cdots \,\leq\,
||v_{A_q}||/||w_{B_q}||  $.
\item[]{\bf (P3)} For $i=1,\ldots,q$, there do not exist
  nontrivial partitions $L_1 \cup L_2$ of $A_i$
and $R_1 \cup R_2$ of $B_i$ 
such that $R_1 \cup L_2 \in \Omega$ and 
$ \,||v_{L_1}||/ ||w_{R_1}||\,< \, ||v_{L_2}||/ ||w_{R_2}||$.
\end{itemize}

The following result is due to Owen and Provan \cite{OP}.
They proved it for the case of BHV tree space.  
The general result is stated in \cite[Corollary 6.19]{MOP}.

\begin{theorem}[Owen-Provan]
\label{thm:owenprovan}
Given points $v,w \in \mathcal{F}_\Omega$ satisfying the hypotheses above,
there exists a unique ordered pair of partitions $({\cal A},{\cal B})$ 
satisfying {\bf (P1)}, {\bf (P2)}, {\bf (P3)}.   The geodesic is
a sequence of $q+1$ line segments,
$G(v,w) =  [v,u^1] \cup [u^1,u^2] \cup \cdots \cup [u^q,w]$.
Its length equals
\begin{equation}
\label{eq:optimallength}
 d(v,w) \,= \,
\sqrt{\mathop{\sum}_{j=1}^{q}{\bigl(
\,||v_{A_j}||\,+\,||w_{B_j}||\, \bigr)^2}}.
\end{equation}
In particular,
$({\cal A},{\cal B})$ is the unique pair of ordered partitions that minimizes (\ref{eq:optimallength}).
The breakpoint $u^i$ lives in the orthant of $\mathcal{F}_\Omega$ that is indexed by
$\bigcup_{j=1}^{i-1} B_j \cup  \bigcup_{j=i+1}^{q}  A_j$.
Setting $u^0=v$, the coordinates
of the breakpoints are computed recursively by the formulas
\begin{equation}\label{bp1}
	\qquad u^{i}_{k}\,\,=\,\,\frac{||u^{i-1}_{A_i}||\cdot w_{k}+||w_{B_i}||\cdot u^{i-1}_{k}}{||u^{i-1}_{A_i}||+||w_{B_i}||}
	\qquad \qquad \qquad \qquad 
	\hbox{for} \,\,\,k\in \bigcup_{j=1}^{i-1}{B_j} ,	
\end{equation}	
\begin{equation}\label{bp2}	
	 \qquad u^{i}_{l}\,=\,\frac{u^{i-1}_{l}}{||u^{i-1}_{A_j}||}\cdot \frac{||w_{B_i}||\cdot||u^{i-1}_{A_j}||-||u^{i-1}_{A_i}||\cdot||w_{B_j}||}{||u^{i-1}_{A_i}||+||w_{B_i}||}
	\quad \hbox{for} \, \,\,\, l\in A_j,\,\, i+1\le j\le q.
\end{equation}
	\end{theorem}

Computing the geodesic between $v$
and $w$ in the orthant space $\mathcal{F}_\Omega$ means identifying
the optimal pair $(\mathcal{A},\mathcal{B})$ among all 
pairs. This is a combinatorial optimization problem.
Owen and Provan \cite{OP} gave a polynomial-time
algorithm for solving this.

We  implemented this algorithm in {\tt Maple}.
Our code works for any flag simplicial
complex $\Omega$ and for
any points $v$ and $w$ in
the orthant space $\mathcal{F}_\Omega$,
regardless of whether they satisfy the hypotheses of
Theorem~\ref{thm:owenprovan}. 
The underlying geometry is as follows:

\begin{theorem}
	\label{prop:breakpoints}
The degenerate cases not covered by Theorem \ref{thm:owenprovan}
are dealt with as follows:
			(1) If $\sigma \cap \tau \ne \emptyset$, 
			then an ordered pair of partitions
			$({\cal A}, {\cal B})$ is constructed 
			 for $(\sigma \backslash \tau,\, \tau \backslash \sigma)$. 
			 The breakpoint $u^{i}$ lives in the orthant of $\mathcal{F}_\Omega$ indexed by $\bigcup_{j=1}^{i-1}{B_j} \cup \bigcup_{j=i+1}^{q}{A_j} \cup (\sigma \cap \tau)$. 
			 It satisfies
\begin{equation}\label{bp3}
	\qquad  u^{i}_{k}\,=\,\,\frac{||u^{i-1}_{A_i}||\cdot w_{k}+||w_{B_i}||\cdot u^{i-1}_{k}}{||u^{i-1}_{A_i}||+||w_{B_i}||}
	\qquad	\hbox{for} \,\,\,k\in \sigma \cap \tau .
\end{equation}
(2) If $|\sigma|<d$ or $|\tau|<d$, i.e.~$v$ or $w$ is in a lower-dimensional orthant of $\mathcal{F}_\Omega$, then we allow one (but not both) of the parts $A_i$ and $B_i$ to be empty,
and we replace {\bf (P2)} by
	\begin{itemize}
		\item[]{\bf (P2')}  $\quad
		||v_{A_i}|| \cdot ||w_{B_{i+1}}|| \,\leq \,
		||w_{B_i}|| \cdot ||v_{A_{i+1}}||\quad $ for $\,\,1\le i\le q-1$.
	\end{itemize}
	We still get an ordered pair  $({\cal A},{\cal B})$ satisfying 
	{\bf (P1)}, {\bf (P2')}, {\bf (P3)}, but it may not be unique.
	 If $i \geq 0$ is the largest index with $A_i=\emptyset$ 
	  and $j \leq q+1$ is the smallest index with $B_j=\emptyset$, 
	  then $G(v,w)$ is a sequence of $j-i-1$  segments. The
	  breakpoints are given in (\ref{bp1}), (\ref{bp2}), (\ref{bp3}). \smallskip \hfill \\
		(3) Now, the length of the geodesic $G(v,w)$ is equal to
	\[
	d(v,w) \,\, = \,\, 
	\sqrt{\mathop{\sum}_{k\in \sigma \cap \tau}{(v_{k}-w_{k})^2}
	\,+\,\mathop{\sum}_{j=1}^{q}{(||v_{A_j}||+||w_{B_j}||)^2}}.\]
\end{theorem}

\begin{proof}
This is an extension of the discussion for BHV tree space in \cite[\S 4]{OP}.
\end{proof}

Our primary object of interest is the  space of ultrametrics $\,\mathcal{U}_m$.
We shall explain its metric structure as an orthant space. It is crucial to 
note that this structure is not unique. Indeed, any polyhedral fan $\Sigma$ in
$\RR^d$ can be refined to a simplicial fan with $n$ rays whose underlying
simplicial complex $\Omega$ is flag. 
 That fan still lives in $\RR^d$, whereas the 
orthant space $\mathcal{F}_\Omega$ 
lives in $\RR^n$.
There is a canonical piecewise-linear isomorphism
between $\mathcal{F}_\Omega$ and the support
$|\Sigma|$, and the Euclidean metric on the former
induces the metric on the latter. The resulting
metric on $|\Sigma|$ depends on the choice of simplicial subdivision.
A different $\Omega$ gives a different metric.

\begin{example}\label{ex:differentomega} \rm
Let $d=2$ and $\Sigma = \RR_{\geq 0}^2$ be the cone
spanned by $v^1 = (1,0)$ and $v^2 = (0,1)$.
In the usual metric on $\Sigma$, the distance between $v^1$ and
$v^2$ is $\sqrt{2}$, and the geodesic passes through $(\frac{1}{2},\frac{1}{2})$.
We refine $\Sigma$ by adding the ray spanned by
$v^3 = (1,1)$. The simplicial complex $\Omega$ has facets
$13$ and $23$, and $\mathcal{F}_\Omega$ is
the fan in $\RR^3$ with maximal cones ${\rm pos}(e_1,e_3)$ and ${\rm pos}(e_2,e_3)$.
The map from this orthant space onto $\Sigma$ induces a different metric.
In that new metric, the distance between $v^1$ and
$v^2$ is $2$, and the geodesic is the {\em cone path} passing through $(0,0)$.
\end{example}

Let $n = 2^m - m -2$. The standard basis vectors in $\RR^n$ are denoted $e_\sigma$,
one for each clade $\sigma $ on $[m]$. Let $\Omega$ be the flag simplicial complex
whose simplices are  the nested sets $\{\sigma_1,\ldots,\sigma_d\}$ as defined in
(\ref{eq:nestedset}). The piecewise-linear isomorphism from the
orthant space $\mathcal{F}_\Omega$~to $\mathcal{U}_m$ takes
the basis vector $e_\sigma$ to the clade metric $D_\sigma $,
and this induces the BHV metric on $\mathcal{U}_m$.
Each orthant in $\mathcal{F}_\Omega$ represents the set of all equidistant trees in $\mathcal{U}_m $ that have a fixed topology.

In Section \ref{sec1},
 each ultrametric $D = (d_{ij}) $ is an element of $\RR^{\binom{m}{2}}\! / \RR \mathbf{1}$.
It has a unique representation
\begin{equation}
\label{eq:Drep}
D \,\, = \,\,
\ell_1 D_{\sigma_1} + 
\ell_2 D_{\sigma_2} +
\cdots + 
 \ell_d D_{\sigma_d},
\end{equation}
 where $\{\sigma_1,\ldots,\sigma_d\} \in \Omega$ and
 $\ell_1,\ldots,\ell_d \geq 0 $. The coefficient $\ell_i$
 is twice the length of the edge labeled $\sigma_i$ in the tree given by $D$.
 It can be recovered from $D $ by the formula
\begin{equation}
\label{eq:recoveredge}
 \ell_i \,\, = \,\,   {\rm min}\bigl\{d_{rt}\,:\,r \in \sigma_i, t \not\in \sigma_i \bigr\}\,\,-\,\,
 {\rm max}\bigl\{d_{rs} \,:\, r,s\in \sigma_i \bigr\} .
\end{equation}
 The maximal nested sets have cardinality $m-2$,
 so this is the dimension of the orthant space $\mathcal{F}_\Omega$.
 The number of maximal nested sets is
$(2m-3)!! = 1\cdot 3 \cdot 5 \cdot \cdots \cdot (2m-3)$.
For instance, Fig.~\ref{fig:eins} shows
that $\mathcal{U}_4$ is consists of
$15$ two-dimensional cones  and $10$ rays.

\begin{figure}[h]
  \begin{center}
    \includegraphics[width=5.5cm]{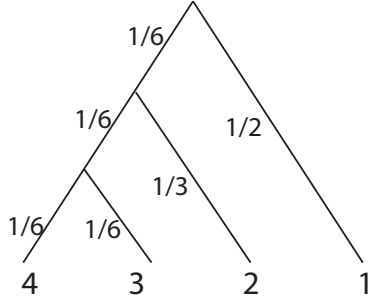} 
  \end{center}
  \caption{The equidistant tree in
    $\mathcal{U}_4^{[1]}$ discussed in Example \ref{eg1}. }\label{fig:exampleT4} 
\end{figure}

\begin{example}\label{eg1} \rm
The equidistant tree in Fig.~\ref{fig:exampleT4} corresponds to the  ultrametric 
\[
D \,=\, (d_{12}, d_{13}, d_{14}, d_{23}, d_{24}, d_{34}) \,=\, 
\bigl(1, 1, 1, \frac{2}{3}, \frac{2}{3}, \frac{1}{3} \bigr) \,\,\, \in \,\,\, \mathcal{U}_4^{[1]}
\,\,\subset \,\, \mathcal{U}_4 .
\]
The clade metrics for the two internal edges are
\[
D_{\sigma_1} = D_{{34}} = (1, 1, 1, 1, 1, 0) \quad \hbox{and} \quad
D_{\sigma_2} = D_{{234}} = (1, 1, 1, 0, 0, 0).
\]
The formula (\ref{eq:recoveredge}) gives
$\ell_1 = \ell_2 = 1/3$. Since all vectors live in $\RR^6$ 
modulo $\RR {\bf 1}$, we have
\[
\ell_1 D_{\sigma_1} + \ell_2D_{\sigma_2} \,=\,
\bigl(\frac{2}{3}, \frac{2}{3}, \frac{2}{3},  \frac{1}{3}, \frac{1}{3}, 0 \bigr)
\,\,= \,\,D.
\]
But, the orthant space that endows $\,\mathcal{U}_4$ with its metric lives in $\RR^{10}$,
and not in $\RR^{6}/\RR {\bf 1}$.
\end{example}

We close this section by reiterating
this extremely important remark.
In Section \ref{sec1} we introduced the set $\,\mathcal{U}_m$ of ultrametrics  as a subset of
a low-dimensional ambient space, having dimension $\binom{m}{2}-1$.
In Section \ref{sec2} we elevated $\mathcal{U}_m$ to live in a
high-dimensional ambient space, having dimension $2^m-m-2$. It is only
the latter realization, as an orthant space, that is used when we 
compute Euclidean distances. In other words, when we compute
geodesics in the BHV metric on tree space, we use the coordinates $\ell_1, 
\ldots, \ell_d$. These are local coordinates on the right-angled cones.
The coordinates $d_{12}, \ldots , d_{ (m-1)m}$ are never to be used
for computing BHV geodesics.

\section{Geodesic Triangles}
\label{sec3}

Biologists are interested in BHV tree space 
as a statistical tool for studying evolution. The geometric
structures in \cite{BHV, MOP} are motivated
by applications such as \cite{Nye}.
This requires a notion of convexity.

We fix a flag simplicial complex $\Omega$ on $[n]$.
The orthant space $\mathcal{F}_\Omega \subset \RR^n$
with its intrinsic Euclidean metric has unique geodesics
$G(v,w)$ as in Theorems \ref{thm:owenprovan} and \ref{prop:breakpoints}.
A subset $T$ of $\mathcal{F}_\Omega $
is called {\em geodesically convex}
if, for any two points $v,w\in T$, 
the unique geodesic $G(v,w)$ is contained in $T$.

Given a subset $S$ of $\mathcal{F}_\Omega$, its
{\em geodesic convex hull} ${\rm gconv}(S)$ is
 the smallest geodesically convex set in $\mathcal{F}_\Omega$
 that contains $S$. 
 If $S = \{v^1,v^2,\ldots,v^s\}$ is a finite set then
 we say that ${\rm gconv}(S)$ is a {\em geodesic polytope}.
If $s=2$ then  we recover the geodesic segment
$\, {\rm gconv}(v^1,v^2) = G(v^1,v^2)$.
If $s=3$ then we obtain a  {\em geodesic triangle}
${\rm gconv}(S) = {\rm gconv}(v^1,v^2,v^3)$.
The main point of this section is to demonstrate
that geodesic triangles are rather complicated objects.

We begin with an iterative scheme for
computing geodesic polytopes.
Let $S$ be any subset of $\mathcal{F}_\Omega$.
Then we can form the union of all geodesics with endpoints in $S$:
	\[ g(S) \,\, = \,\, \bigcup_{v,w \in S} {G(v,w)}. \]
	For any integer $t\ge 1$, define $g^{t}(S)$ recursively as
        $g^{t}(S)=g(g^{t-1}(S))$, with $g^0(S) = S$.

\begin{lemma}
\label{lem:moreandmore}
	Let $S$ be a set of points in $\mathcal{F}_\Omega$. Then 
	its geodesic convex hull equals
	\[gconv(S)\,\,=\,\,\bigcup_{t=0}^{\infty}{g^{t}(S)}.\]
	In words, the geodesic convex hull of $S$ is the set of all points in 
	$\mathcal{F}_\Omega$ that can be generated in 
	finitely many steps from $S$ by taking geodesic segments.
\end{lemma}

\begin{proof}
If $T$ is a subset of ${\rm gconv}(S)$ then
 $g(T) \subseteq {\rm gconv}(S)$. By induction on $t$, we see that
	$\,g^{t}(S)\subseteq {\rm gconv}(S)$ for all $t \in \mathbb{N}$. Therefore,
	 $\bigcup_{t=0}^{\infty}{g^{t}(S)}\subseteq {\rm gconv}(S)$.
	  On the other hand, for any two points $v,w\in \bigcup_{t=0}^{\infty}{g^{t}(S)}$, there exist positive integers $t_1,t_2$ such that $v\in g^{t_1}(S)$
	  and $w \in g^{t_2}(S)$. The geodesic path $G(v,w)$ is contained in
	  $g^{\max(t_1,t_2)+1}(S)$, so $G(v,w)$ is in 
	$   \bigcup_{t=0}^{\infty}{g^{t}(S)}$. So, this set is geodesically convex, and
	   we conclude that it equals ${\rm gconv}(S)$.
\end{proof}

Lemma \ref{lem:moreandmore} gives
a numerical method for approximating 
geodesic polytopes by iterating the
computation of geodesics. However, it is not clear
whether this process converges.
The analogue for negatively curved continuous spaces 
arises in \cite[Lemma 2.1]{FMPV}, along with a pointer to the
following open problem stated in \cite[Note 6.1.3.1]{Ber}:

\smallskip

{\em ``An extraordinarily simple question is still open (to the best of
our knowledge). What is the convex envelope of three points in a 3- or higher-dimensional Riemannian manifold? We look for the smallest possible set which
contains these three points and which is convex. For example, it is unknown
if this set is closed. The standing conjecture is that it is not closed, except in
very special cases, the question starting typically in $\mathbb{CP}^2$. The only text we
know of addressing this question is {\rm \cite{Bow}}.''}  \hfill
It seems that this question is also open in our setting:

\begin{problem} \label{prob:berger}
Are geodesic triangles  in orthant spaces
$\mathcal{F}_\Omega$ always closed?
\end{problem}

If $T$ is a geodesically convex subset of $\mathcal{F}_\Omega$
then its restriction $T_\sigma \, = \,T \cap \mathcal{O}_\sigma$
to any orthant is a convex set in the usual sense.
If Problem~\ref{prob:berger} has an affirmative answer
then one might further conjecture that each geodesic polytope~$T$
is a polyhedral complex with cells $T_\sigma$.
This holds in the examples we computed, but the
matter is quite subtle. The segments of the pairwise geodesics need not
be part of the complex $\{T_\sigma\}$, as the following example shows.

\begin{example} 
\label{ex:threepages} \rm
Consider a $2$-dimensional orthant space
that is locally an open book \cite{openbook}
with three pages. We pick three points  $a,b,c$ on these pages
as shown in Fig.~\ref{fig:subtle} and~\ref{fig:beauty}.
The pairwise geodesics $G(a,b)$, $G(a,c)$ and $G(b,c)$ 
determine a set that is not a polyhedral complex unless one
triangle is subdivided. That set, shown on the left in Fig.~\ref{fig:subtle},
is not geodesically convex. We must enlarge it to get
the geodesic triangle ${\rm gconv}(a,b,c)$,
shown on the right in Fig.~\ref{fig:subtle}.
It consists of three classical triangles, one in each page of the book.
Note that the geodesic from $a$ to $c$ travels
through the interiors of two classical triangles.
\end{example}

\begin{figure}[h!]
  \begin{center}
    \includegraphics[width=5.5cm]{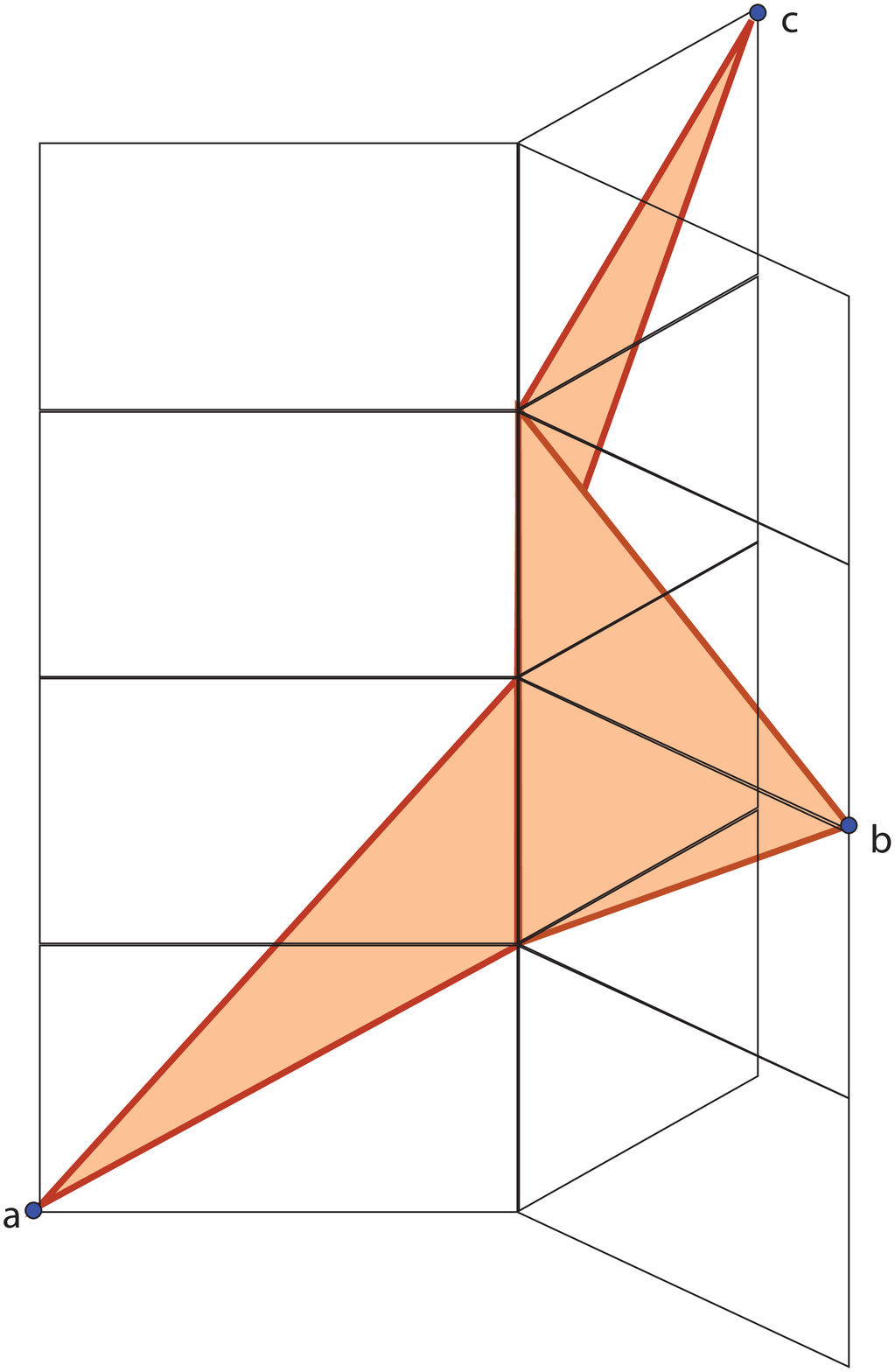} \quad
        \includegraphics[width=5.5cm]{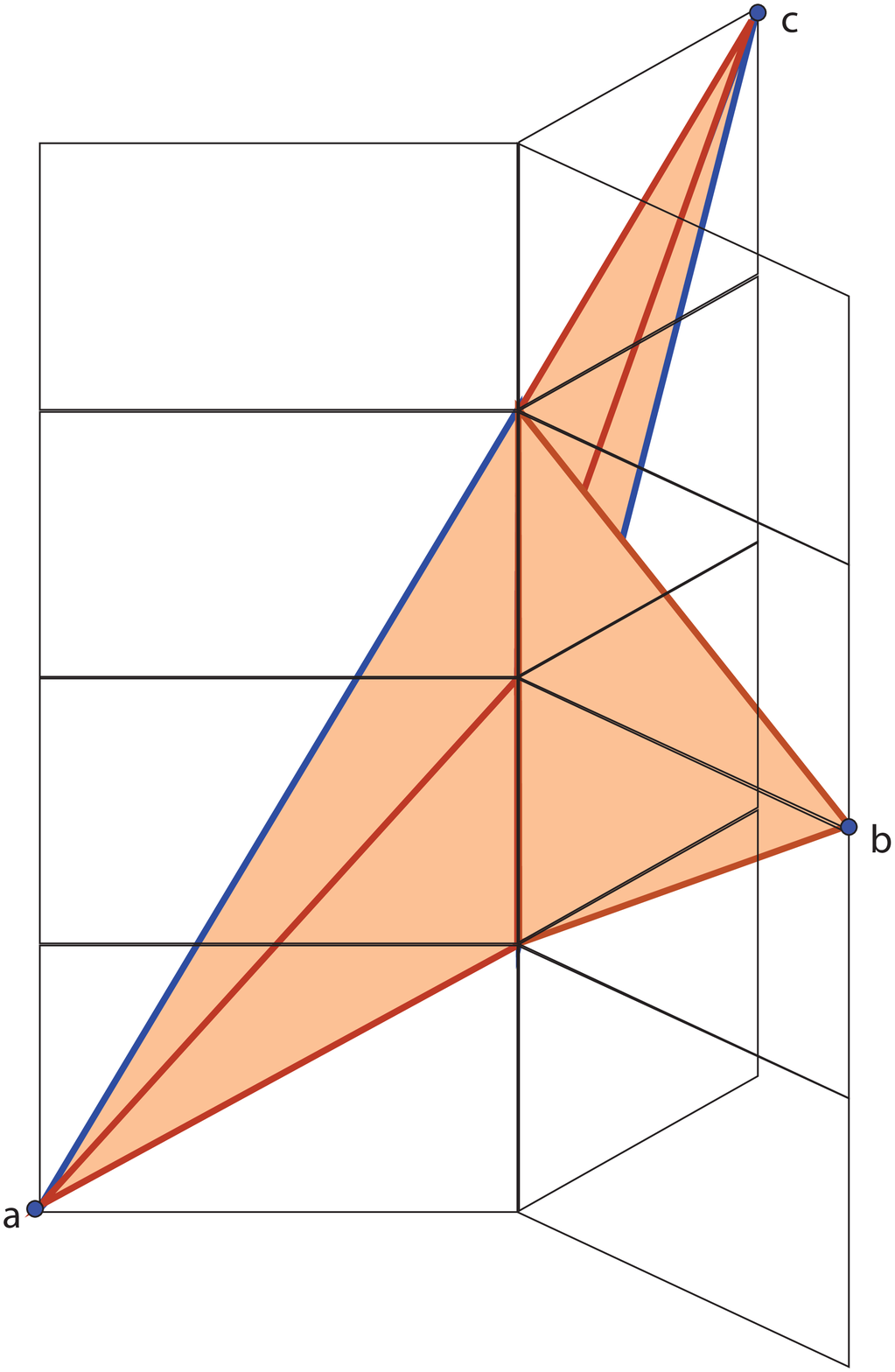}
  \end{center}
  \caption{\label{fig:subtle} 
  Three points  in the open book with three pages.
  Their geodesic convex hull is shown on the right.
  The left diagram shows the subset formed
  by the pairwise geodesics.
    } \end{figure}

We next present a sufficient condition for a set $T$ to be
geodesically convex. We regard each orthant 
$\mathcal{O}_\sigma \simeq \RR^d_{\geq 0} $ as a poset 
by taking the component-wise partial order.

\begin{theorem} \label{lem:geocon}
Let $\,T$ be a subset of an orthant space $\mathcal{F}_\Omega$
such that,  for each simplex $\sigma \in \Omega$,
the restriction $T_\sigma$ is both convex and  an order ideal in $\mathcal{O}_\sigma$.
Then $\,T$ is geodesically convex.
\end{theorem}

\begin{proof}
Let $v\in T_\sigma$, $w\in T_\tau$ and $G(v,w) =  [v,u^1] \cup [u^1,u^2] \cup \cdots \cup [u^q,w]$. 
In order to prove $G(v, w)\subset T$,  it suffices to show $u^i\in T$ 
for all $i$, since the restriction of $T$ 
 to each orthant is convex.
 We first prove  $u^1\in T$ by constructing a point $u^*\in T_{\sigma}$ such that $u^1\leq u^*$.
We let
\[u^*\,\,=\,\,
\lambda v_{\sigma\backslash A_1} + (1-\lambda ) w_{\sigma\cap \tau}, \;\;\; \text{where}\; 
\,\,\lambda\,=\, \frac{||w_{B_1}||}{||v_{A_1}||+||w_{B_1}||}.\]
Since the restriction of $T$ to each orthant is an order ideal, 
we know $v_{\sigma\backslash A_i}\leq v\in T_{\sigma}$. We also have $w_{\sigma\cap \tau}\leq w\in T_{\tau}\subset T$ and hence $w_{\sigma\cap \tau}\in T_{\sigma}$.  
Thus, $u^*\in T_{\sigma}$ since $T_\sigma$ is convex.
By  formula (\ref{bp2}), 
$u^*_k =u^1_k + (1-\lambda)\frac{||w_{B_j}||}{||v_{A_j}||}v_k $
 if $k\in A_j$
   for some $j=2, \ldots, q$. By  formula (\ref{bp1}), 
$u^*_k = u^{1}_k$  if $k\in {\sigma}\backslash (\bigcup_{j=2}^q A_j) $.
Since $\lambda\leq 1$, we conclude  $u^1\leq u^*$.
This implies $u^1\in T_{\sigma} \subset T$. 
By a similar argument,  $u^i\in T$ since $u^{i}$ is the first breakpoint of $G(u^{i-1}, w)$ for every $i=2, \ldots, q$. 
 \end{proof}

\begin{corollary}
The compact tree space $\,\mathcal{U}_m^{[1]}$ is geodesically convex.
\end{corollary}

\begin{proof}
If  $T = \mathcal{U}_m^{[1]}$ then $T_\sigma$ is a subpolytope of 
the cube $[0,1]^d$, with coordinates $\ell_1,\ldots,\ell_d$ as in the
end of Section \ref{sec1}.
The polytope $T_\sigma$ is an order ideal because decreasing
 the edge lengths in a phylogenetic tree can only decrease
 the distances   between pairs of leaves.
\end{proof}

The sufficient condition in Theorem \ref{lem:geocon}
is far from necessary. It is generally hard to
 verify that a  set $T$ is geodesically convex,
even if $\{T_\sigma\}_{\sigma \in \Omega}$ is a polyhedral complex.

\begin{example}[A Geodesic Triangle] \rm
Let $\Omega$ be the $2$-dimensional flag complex with
facets $123$, $234$, $345$ and $456$.
We consider eight points in the 
$3$-dimensional orthant space $\mathcal{F}_\Omega \subset \RR^6$:
$$ \begin{matrix}
\,\,a=(4, 6, 6, 0, 0, 0),\,\, b=(0, 5, 8, 0, 0, 0), \,\,c=(0, 0, 0, 1, 2, 3),\,\,\,
d=(0, 0, \frac{1}{7}, \frac{5}{7}, 0, 0),  \\
 e=(0, 0, 0, \frac{8}{11}, \frac{1}{11}, 0),\, f=(0, 0, 0, \frac{14}{25}, 0, 0),\,
h=(0, \frac{14}{19}, \frac{14}{19}, 0, 0, 0), \,\,
o=(0, 0, 0, 0, 0, 0). \end{matrix} $$
The geodesic triangle $T  = {\rm gconv}(a,b,c)$
is a $3$-dimensional polyhedral complex. Its
maximal cells are 
three tetrahedra $T_{123} = {\rm conv}\{a,b,h,o\}$,
$\,T_ {345} = {\rm conv}\{d,e,f,o\}$,
$T_{456} = {\rm conv}\{c,e,f,o\}$,
and one bipyramid
$ T_{234} = {\rm conv}\{b,d,f,h,o\}$.
These four $3$-polytopes are attached along the triangles
$T_{23} = {\rm conv}\{b,h,o \}$,
$T_{34} = {\rm conv}\{ d,f,o\}$, and
$T_{45} = {\rm conv}\{ e,f,o\}$.

Verifying this example amounts to a non-trivial computation.
We first check that $d,e,f,h,o$ are in the geodesic 
convex hull of $a,b,c$.
This is done by computing 
geodesic segments as in Section \ref{sec2}. We find
$G(a,c) = [a,o] \cup [o,c]$
and $G(b,c) = [b,d] \cup [d,e] \cup [e,c]$.
Next, 
$G(a,e)=[a, x]\cup [x, y]\cup [y, e]$ where $x=(0, \frac{11}{13}, \frac{12}{13}, 0, 0, 0)$ and
 $y=(0, 0, \frac{6}{67}, \frac{44}{67}, 0, 0)$. 
 If we now take $s=\frac{9}{22}b+\frac{13}{22}x$
 in $\mathcal{O}_{23}$ then   $G(s, c)=[s, f]\cup [f, c]$,
 and finally $G(a,f) = [a,h] \cup [h,f]$.
Now, we need to show that the union of the four
$3$-polytopes is geodesically convex.    
For any two points $v$ and $w$ from distinct polytopes
we must show that $G(v,w)$ remains in that union.
Here, it does not suffice to take vertices. 
This is a  quantifier elimination problem
in piecewise-linear algebra, and we are proud to report
that we completed this computation.
\end{example}

The following theorem is the main result in this section.

\begin{theorem} \label{thm:highdim}
Let $d$ be a positive integer.
There exists an orthant space of dimension $2d$
and three points in that space such that their geodesic triangle
contains a $d$-dimensional simplex.
\end{theorem}


\begin{proof}
Fix the simplicial complex $\Omega$ on the vertex set $\{1,2,\ldots,4d\}$
whose $2d+1$ facets are $\{i+1,i+2,\ldots,i+2d \}$ for $i=0,1,\ldots,2d$.
This simplicial complex is flag because the minimal non-faces are the 
$\binom{2d+1}{2}$ pairs $\{i,j\}$ for $j-i \geq 2d$.
The corresponding orthant space ${\cal F}_{\Omega}$ has dimension $2d$.
We here denote the maximal orthants in ${\cal F}_\Omega$ by
$\mathcal{O}_i = {\rm pos}\{ e_{i+1},e_{i+2}, \ldots,e_{i+2d}\}$.

For each positive integer $i$, we define an integer
$v_i$ as follows. 
We set $v_i = \frac{i}{2}$  if $i$ is even and $v_i =\frac{7(i+1)}{2}$ if $i$ is odd.
We fix the points
$a = \sum_{i=1}^{2d} v_i e_i$ and
$b = \sum_{i=1}^{2d} e_i$ in the first orthant $\mathcal{O}_0$
and the point $c = \sum_{i=1}^{2d} e_{2d+i}$ in 
the last orthant $\mathcal{O}_{2d}$.  Explicitly, 
$$ \begin{array}{lll}
a &=&  (\,7, 1, 14, 2, 21, 3, 28, 4, \,\ldots, \,
  0, 0, 0, 0, 0, 0, 0, 0,  \ldots ),
  \\
b &=&  (\,1, \,1, \,1, \,1, \,1, \,1, \,1,\, 1, \,\,\ldots, \,
 0, 0, 0, 0, 0, 0, 0, 0, \ldots ), \\
c  &=& (\,0,\, 0,\, 0,\, 0,\, 0, \,0,\, 0,\, 0, \, \ldots\,, \,
 1,1,1,1,1,1,1,1,\,\ldots).
\end{array}
$$
Consider the geodesic triangle 
${\rm gconv}(a,b,c)$ in the orthant space ${\cal F}_{\Omega}$. 
We shall construct a simplex $P$ of dimension $d$ that
 is contained in the convex set ${\rm gconv}(a,b,c) \cap  \mathcal{O}_0$.
 
We begin with the geodesic segment ${G}(a,c)$. The pair of ordered partitions is 
\begin{equation}
 \label{eq:partitionsAC} 
\begin{array}{lllll}
 \bigl( ( \{1,2\}, &\{3,4\}, & \ldots, &\{2d-1,2d\} ), &\, \\
   \;\;(\{2d{+}1,2d{+}2\},& \{2d{+}3,2d{+}4\}, & \ldots, & \{4d-1,4d\} ) \bigr).
\end{array}
\end{equation}
We write the corresponding decompositions into classical line segments as follows:
$$
\begin{matrix}
G(a,c) & = & [a,u^1] \,\cup\, [u^1,u^2] \,\cup\, \cdots \,\cup \,[u^{d-1}, u^d] 
\cup [u^d, c]. \\
\end{matrix}
$$
Each $u^i$, for $1 \leq i \leq d$,
 lies in the relative interior of an orthant of dimension $2d-2$:
 $$
\begin{matrix}
u^i~ \in~ {\rm pos} \{e_{2i-1},\ldots, e_{2d + 2i -2} \} \cap {\rm pos} \{e_{2i+1},\ldots, e_{2d + 2i} \}~ =~ {\rm pos} \{e_{2i+1},\ldots, e_{2d + 2i -2} \} \\
\end{matrix}.
$$
We now consider the geodesic segments
${G}(b,u^i)$
for $i = 1,2,\ldots,d$.
Let $\widetilde u^i$ denote the
unique intersection point of these geodesic segments
with the boundary of $\mathcal{O}_0$.
Note that $\widetilde u^1 = u^1$.
The points $a,\widetilde u^1, \ldots, \widetilde u^d$ 
lie in the   orthant $\mathcal{O}_0 \simeq \RR^{2d}_{\geq 0}$.
By construction, they are also contained in the geodesic triangle
${\rm gconv}(a,b,c) $. We shall prove that they are affinely independent.

The above point $u^i$ 
 can be written as  $\sum_{k=2i-1}^{2d+ 2i -2} u_k^i e_k$ or
as $\sum_{k=2i+1}^{2d+ 2i} u_k^i e_k$. Note that $u_k^i=0$ for  $k =2i{-}1, 2i, 2d{+}2i{-}1, 2d{+}2i$. For 
$2i+1\leq k\leq 2d+2i-2$, 
we
claim: 
\begin{eqnarray}    
	u^{i}_{2j-1}~=~\frac{7(j-i)}{5i+1},\;\;\;\; 
	u^{i}_{2j}&=&\frac{j-i}{5i+1},\;\;\;\;\text{for}\;i+1\leq j\leq d;\; \label{eq:firstnodes}\\
	u^{i}_{2d+2j-1}~=~u^{i}_{2d+2j}&=&\frac{5(i-j)}{5i+1},\;\;\text{for}\;1\leq j\leq i-1. \label{eq:lastnodes}
\end{eqnarray}

\noindent
To prove the above claim, it suffices to verify that the proposed $u^i$ satisfy
\[\sum_{i=0}^{d}{d(u^{i},u^{i+1})}~=~d(a,c),\;\; \text{where}\;u^0=a,u^{d+1}=c.\]
In fact, by (\ref{eq:optimallength}) we have 
\[d(a,c)~=~\sqrt{\sum_{i=1}^{d}{(5i\sqrt{2}+\sqrt{2})^2}}~=~\sqrt{2 S_d}
\qquad \hbox{where} \quad S_d = \sum_{i=1}^{d}{(5i+1)^2}. \]
For $0\le i\le d-1$, we have 
\begin{eqnarray*}
d(u^{i},u^{i+1})
&=&\sqrt{\sum_{k=2i+1}^{2d+2i}{(u^{i}_{k}-u^{i+1}_{k})^2}} \\
&=&\sqrt{\sum_{j=0}^{d-i-1}{(7^2+1^2)(\frac{j+1}{5i+1}-\frac{j}{5i+6})^2}+\sum_{j=0}^{i-1}{2(\frac{5j}{5i+1}-\frac{5j+5}{5i+6})^2}} \\
&=&\frac{\sqrt{50}}{(5i+1)(5i+6)}\sqrt{\sum_{j=i+1}^{d}{(5j+1)^2}+\sum_{j=1}^{i}{(5j+1)^2}} \\
&=&\frac{5}{(5i+1)(5i+6)}\sqrt{2S_d}.
\end{eqnarray*}
The sum of these $d$ quantities simplifies to $\,\bigl(1-\frac{1}{5d+1}\bigr)\sqrt{2S_d}$.
We next observe that
\begin{eqnarray*}
d(u^d,c)&=&\sqrt{\sum_{k=2d+1}^{4d}{(u^d_{k}-c_{k})^2}} 
~=~\sqrt{2\sum_{j=0}^{d-1}{(\frac{5j}{5d+1}-1)^2}} 
~=~\frac{\sqrt{2S_d}}{5d+1}.
\end{eqnarray*}
By adding this to the previous sum, we obtain $\sqrt{2S_d}=d(a,c)$.
So, the claim is proved.

Next, we compute $\widetilde u^i$ for $1\le i\le d$.
Recall $\widetilde u^i$ is the unique intersection point of the geodesic segment
${G}(b,u^i)$ with ${\rm pos} \{e_3, \ldots, e_d\}$. 
We note two facts about ${G}(b,u^i)$:
\begin{itemize}
\item[(F1)]~The common coordinates of $b$ and $u^i$ are $e_{2i-1}, \ldots, e_{2d}$. 
\item[(F2)]~By equation (\ref{eq:lastnodes}), the pair of ordered partitions determining $G(b, u^i)$ is 
\begin{equation*}
 \label{eq:partitionsABC} 
\begin{array}{llllll}
 \bigl( \,(\{1\}, & \{2\},  &\ldots, &\{2i-2\} ),&\, \\
       \; \;\;(\{2d{+}1\}, &\{2d{+}2\}, & \ldots,& \{2d+2i-2\} ) \bigr).
\end{array}
\end{equation*}
\end{itemize}
Suppose $\widetilde u^i=\sum_{k=1}^{2d} \widetilde u^i_k e_k$.  For $1\leq k\leq 2i-2$, we 
compute $\widetilde u^i_k$ by (\ref{bp2}) and (\ref{eq:firstnodes}). For $2i-1\leq k\leq 2d$, 
we compute $\widetilde u^i_k$ by 
 (\ref{bp3}) and (\ref{eq:lastnodes}). Then we obtain $\widetilde u^i_k$ as follows:
 
\begin{equation}
	\widetilde u^{i}_{2j}=
	\begin{cases}
	\frac{j+4i-5}{10i-4}, &\text{ if } i+1\le j\le d;\\
	\frac{5j-5}{10i-4}, &\text{ if } 1\le j\le i\\
	\end{cases}
\end{equation}
and
\begin{equation}
	\widetilde u^{i}_{2j-1}=
	\begin{cases}
		\frac{7j-2i-5}{10i-4}, &\text{ if } i+1\le j\le d;\\
		\frac{5j-5}{10i-4}, &\text{ if } 1\le j\le i.\\
	\end{cases}
\end{equation}

The  $d+1$ points
$\,\widetilde u^0  = a,\widetilde u^1, \ldots, \widetilde u^d\,$ are contained in
$\mathcal{O}_0 \simeq \RR^{2d}$ also in our geodesic triangle.
To complete the proof of Theorem \ref{thm:highdim}, we will now show that
they are affinely independent, so their convex hull is a $d$-simplex.
 Consider the
$(d+1)\times (2d+1)$ matrix $U$, whose $(i+1)$-th row is the 
vector of homogeneous coordinates of $\widetilde u^i$.  
We must  show that $U$ has rank $d+1$. 
Let $U'$ be the integer matrix obtained from $U$ by multiplying
each row by the denominator $10i-4$.
Let $U''$ be the $(d+1)\times (d+1)$ submatrix of $U'$ formed by the $2,4,\ldots,2d$-th and $(2d+1)$-th columns of $U'$. We apply elementary  column operators to $U''$ to obtain 
a triangular form. From this, we find that
$|\det(U'')|=4^{d-3}|10d-34|\ne 0$. This means that $U''$ has rank $d+1$,
and hence so do the rectangular matrices $U'$ and $U$. 

For an example take $d=5$.
The matrix representing
 $a, \widetilde u^1, \widetilde u^2 ,\widetilde u^3, \widetilde u^4, \widetilde u^5
 $ in $\RR^{10}$ is
 $$ \!\!
U~=~\left[ \begin {array}{ccccccccccc} 
7&1&14&2&21&3&28&4&35&5&1 \\ 
\noalign{\medskip}0&0&\frac{7}{6}&\frac{1}{6}&\frac{14}{6}&\frac{2}{6}&\frac{21}{6}&\frac{3}{6}&\frac{28}{6}&\frac{4}{6}&1
\\ 
\noalign{\medskip}0&0&\frac{5}{16}&\frac{5}{16}&\frac{12}{16}&\frac{6}{16}&\frac{19}{16}&\frac{7}{16}&\frac{26}{16}&\frac{8}{16}&1
\\  
\noalign{\medskip}0&0&\frac{5}{26}&\frac{5}{26}&\frac{10}{26}&\frac{10}{26}&\frac{17}{26}&\frac{11}{26}&\frac{24}{26}&\frac{12}{26}&1
\\
\noalign{\medskip}0&0&\frac{5}{36}&\frac{5}{36}&\frac{10}{36}&\frac{10}{36}&\frac{15}{36}&\frac{15}{36}&\frac{22}{36}&\frac{16}{36}&1
\\ 
\noalign{\medskip}0&0&\frac{5}{46}&\frac{5}{46}&\frac{10}{46}&\frac{10}{46}&\frac{15}{46}&\frac{15}{46}&\frac{20}{46}&\frac{20}{46}&1
\end {array} \right]_{6\times 11}.
$$
This matrix has rank $6$, so its columns form a $5$-simplex.
 Hence, the geodesic triangle spanned
by the points $a,b,c$ in the orthant space ${\cal F}_{\Omega}$ 
has dimension at least $5$.
\end{proof}

A nice feature of the construction above is that it extends to
 BHV tree space.
The geodesic convex hull of three equidistant trees
can have arbitrarily high dimension.

\begin{corollary}  \label{cor:treebig}
There exist three ultrametric trees with $2d+2$ leaves
whose geodesic triangle in ultrametric BHV space 
$\mathcal{U}_{2d+2}$ has dimension at least $d$.
\end{corollary}

\begin{proof}
Consider the following sequence of $4d$ clades
on the set $[2d+2]$. Start with the $2d$ clades
$\{1,2\}, \{1,2,3\}, \{1,2,3,4\}, \ldots, \{1,2,3,\ldots,2d+1\}$.
Then continue with the $2d$ clades
$\{1,3\}, \{1,3,4\}, \ldots, \{1,3,4,5,\ldots,2d+2\}$.
For instance, for $d=5$, our sequence equals
$$ \begin{array}{ll}
&\{1,2\},
\{1,2,3\},
 \{1,2,3,4\},
 \{1,2,3,4,5\},
 \dots,
 \{1,2,3,4,5, 6, 7, 8, 9, 10, 11\} ,\\
&  \{1,3\},
   \{1,3,4\},
    \{1,3,4,5\},
    \{1,3,4,5,6\},
    \dots,
     \{1,3,4,5,6,7, 8, 9, 10, 11, 12\}.
     \end{array} 
 $$
 
     In this sequence of $4d$ clades, every collection
of $2d$ consecutive splits is compatible
and  forms a trivalent caterpillar tree.
No other pair is compatible.
Hence the induced subfan of the tree space $\mathcal{U}_{2d+2}$
is identical to the orthant space ${\cal F}_{\Omega}$ in the proof above.
The two spaces are isometric.
Hence our high-dimensional geodesic triangle exists also in tree space
$\mathcal{U}_{2d+2}$.
\end{proof}

\section{Tropical Convexity}
\label{sec4}

In this section we shift gears, by turning to tropical convexity.
We shall assume that the reader is familiar with basics of
tropical geometry \cite{MS}. We here use
the max-plus algebra, so our convention is opposite to that of  \cite{MS, PS}.
The connection between phylogenetic trees and tropical lines,
identifying tree space with a tropical Grassmannian, has
been explained in many sources, including 
\cite[\S 4.3]{MS} and \cite[\S 3.5]{PS}.
However, the restriction to
ultrametrics \cite[\S 4]{AK} offers a fresh perspective.
From that vantage point, the discussion of tree mixtures
at the end of \cite[\S 3.5]{PS} seems to be misleading.
We posit here:~{\em mixtures of trees are trees!}

Let $L_m$ denote the cycle space of the complete  graph
on $m$ nodes with $e := {m \choose 2}$ edges.  
 This is the $(m-1)$-dimensional subspace  of $\RR^{e}$ defined
by the linear equations $\,x_{ij} - x_{ik} + x_{jk} = 0\,$
for $1 \leq i < j < k \leq m$.
The tropicalization of the linear space $L_m$ is
the set of points  $D = (d_{ij})$ such that
the following maximum is attained at least twice
for all triples $i,j,k$:
$$ d_{ij} \oplus d_{ik} \oplus d_{jk} \, = \,  {\rm max}(d_{ij}, d_{ik},d_{jk}). $$
Disregarding nonnegativity constraints and
triangle inequalities, points in ${\rm Trop}(L_m)$ are
precisely the ultrametrics on $[m]$;
see \cite[Theorem 3]{AK} and \cite[Example 4.2.14]{MS}.

As is customary in tropical geometry, we work in the quotient space
$\RR^{e}/\RR {\bf 1}$, where ${\bf 1} = (1,1,\ldots,1)$. 
The images of ${\rm Trop}(L_m)$
and $\mathcal{U}_m$ in that space are equal. Each
point in that image has a unique representative
whose coordinates have tropical sum
$\, d_{12} \oplus d_{13} \oplus \cdots \oplus d_{m-1,m}
= {\rm max}(d_{12},d_{13} \ldots,d_{m-1,m}) \,$
equal to $1$.
Thus, in $\RR^{e}/\RR {\bf 1}$,
$$ \mathcal{U}_m^{[1]} \,\, \subset \,\,
\mathcal{U}_m \, = \, {\rm Trop}(L_m). $$

Given two elements $D$ and $D'$ in 
$\RR^{e}$, their {\em tropical sum}
$D \oplus D'$ is the coordinate-wise maximum. A subset of $\RR^{e}$
is {\em tropically convex} if it is closed under the
tropical sum operation. The same definitions apply
to elements and subsets of  $\,\RR^{e}/\RR {\bf 1}$.

\begin{proposition}
The tree space $\mathcal{U}_m$
is a tropical linear space and is hence tropically convex.
The compact tree space $\,\mathcal{U}_m^{[1]}$
is a tropically convex subset.
\end{proposition}

\begin{proof}
We saw that $\mathcal{U}_m  = {\rm Trop}(L_m)$ is a tropical linear space,
so it is tropically convex by \cite[Proposition 5.2.8]{MS}.
We show that its subset $\mathcal{U}_m^{[1]}$ is closed under  tropical sums.
Suppose  two real vectors $a = (a_1,\ldots,a_e)$ and
$b = (b_1,\ldots,b_e)$ satisfy $|a_i - a_j| \leq 1$ and $|b_i - b_j| \leq 1$
for all $i,j$. Then the same holds for their tropical sum $a \oplus b$.
Indeed, let $a_i \oplus b_i$ be the largest coordinate of $a \oplus b$
and let $a_j \oplus b_j$ be the smallest. There are four cases as
to which attains the two maxima given by $\oplus$.
 In all four cases, one easily checks that
$|(a_i \oplus b_i) - (a_j \oplus b_j) |\leq 1$.
\end{proof}

We briefly recall some basics from tropical convexity \cite[\S 5.2]{MS}.
A {\em tropical segment} is the tropical convex hull 
${\rm tconv}(u,v)$ of two 
points $u = (u_1{, }u_2{, }\ldots{, }u_e)$ and 
$v = (v_1{, }v_2{, }\ldots {, }v_e)$ in $\RR^e/\RR {\bf 1}$.
It is the concatenation of at most $e-1$ ordinary line segments,
with slopes in $\{0,1\}^e$. Computing that segment 
involves sorting the coordinates of $u-v$, so it is done in
time $O(e \cdot {\rm log}(e))$.  This algorithm is described in the proof
of \cite[Proposition 5.2.5]{MS}.

A {\em tropical polytope} $\mathcal{P} = {\rm tconv}(\mathcal{S})$ is the tropical 
convex hull of a finite set $\mathcal{S}$ in  $\RR^e/\RR {\bf 1}$.
This is a classical polyhedral complex of dimension at most $ |\mathcal{S}| - 1$.

If $\mathcal{D}$ is a real $s \times e$-matrix then the tropical convex hull of 
its $s$ rows is a tropical polytope in $\RR^e/\RR {\bf 1}$.
The tropical convex hull of its $e$ columns is a tropical polytope
in $\RR^s/\RR {\bf 1}$. It is a remarkable fact \cite[Theorem 5.2.21]{MS}
that these two tropical polytopes are identical.
We write  ${\rm tconv}(\mathcal{D})$ for that common object.
Example~\ref{ex:threebyten} illustrates this for
a $3 {\times} 10$-matrix $\mathcal{D}$. Here, ${\rm tconv}(\mathcal{D})$
has dimension $2$ and is shown in Fig.~\ref{fig:threebyten}.

\begin{example} \label{ex:TropTriangle} \rm
\label{ex:threebyten} We compute the tropical convex hull
${\rm tconv}(\mathcal{D})$ of the $3 \times 10$-matrix
$$
{\small
\mathcal{D} \,\, = \,\,
\begin{pmatrix}
 77/100 & 1 &  21/25 &     1 &     1 &  21/25 &   1 &     1 &  13/25 &     1  \\
      1 & 1 &     1 &     1  &   8/25 &    3/4 & 3/4 &   3/4 &   3/4 & 23/50  \\
     1 & 1 &      1 & 49/50 & 16/25 & 16/25 &   1 & 3/100 &     1 &     1 
\end{pmatrix}.
}
$$

\begin{figure}[h!]
  \begin{center}
    \includegraphics[width=8cm]{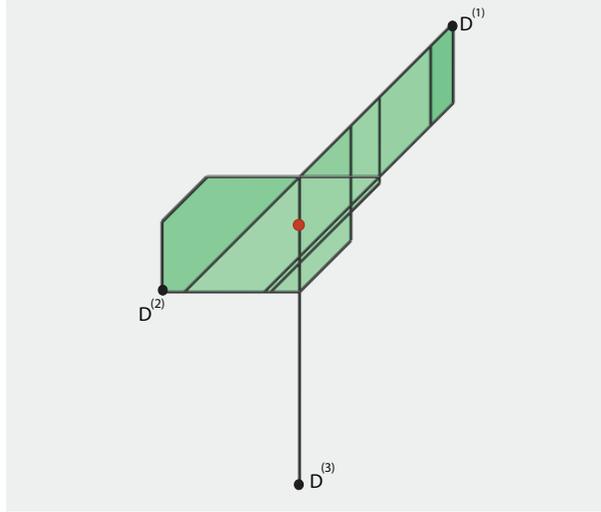}
  \end{center}
  \caption{\label{fig:threebyten} The tropical triangle
  formed by three equidistant trees  on five taxa.    } \end{figure}
  
This is  the tropical convex hull
of the $10$ points in the plane $\RR^3/\RR {\bf 1}$
represented by the column vectors. Its type decomposition
is a $2$-dimensional polyhedral complex with 
$23$ nodes, $35$ edges and $13$ two-dimensional cells. 
It is shown in Fig.~\ref{fig:threebyten}.
We can also regard ${\rm tconv}(\mathcal{D})$ as a tropical triangle
in $\RR^{10}/\RR {\bf 1}$, namely as the tropical
convex hull of the row vectors.
The three rows of $\mathcal{D}$ are ultrametrics
$(d_{12},d_{13},d_{14},d_{15},d_{23},d_{24},d_{25}, d_{34}, d_{35},
d_{45})$, i.e.~points in $\mathcal{U}_5$.   We denote the three row vectors of 
the matrix $\mathcal{D}$ by $D^{(1)}$, $D^{(2)}$ and $D^{(3)}$, respectively. 

\begin{table}[t]
\begin{center}
\setcounter{MaxMatrixCols}{12}
{\tiny
$$
\begin{matrix}
  d_{12} & d_{13} & d_{14} & d_{15} & d_{23} & 
  d_{24} & d_{25} & d_{34} & d_{35} & d_{45} & {\rm nested} \,\,{\rm set} \\
 77/100 & 1 & 21/25 & 1 & 1 & 21/25 & 1 & 1 & 13/25 & 1 & \{12,35,124\}  \bullet  \\
 21/25 & 1 & 21/25 & 1 & 1 & 21/25 & 1 & 1 & 59/100 & 1 & \{35,124\}  \\
 1 & 1 & 1 & 1 & 1 & 21/25 & 1 & 1 & 3/4 & 1 & \{24,35\}  \\
 1 & 1 & 1 & 1 & 91/100 & 3/4 & 91/100 & 91/100 & 3/4 & 91/100 & \{24,35,2345\}  \\
 1 & 1 & 1 & 1 & 3/4 & 3/4 & 3/4 & 3/4 & 3/4 & 3/4 & \{2345\}    \\
 1 & 1 & 1 & 1 & 23/50 & 3/4 & 3/4 & 3/4 & 3/4 & 23/50 & \{23,45,2345\}   \\
 1 & 1 & 1 & 1 & 8/25 & 3/4 & 3/4 & 3/4 & 3/4 & 23/50 & \{23,45,2345\} \bullet  \\
 1 & 1 & 1 & 1 & 8/25 & 3/4 & 3/4 & 3/4 & 3/4 & 17/25 & \{23,45,2345\}   \\
 1 & 1 & 1 & 1 & 39/100 & 3/4 & 3/4 & 3/4 & 3/4 & 3/4 & \{23,2345\}   \\
 1 & 1 & 1 & 1 & 16/25 & 3/4 & 1 & 3/4 & 1 & 1 & \{23,234\}   \\
 1 & 1 & 1 & 49/50 & 16/25 & 73/100 & 1 & 73/100 & 1 & 1 & \{15,23,234\}   \\
 1 & 1 & 1 & 49/50 & 16/25 & 16/25 & 1 & 3/100 & 1 & 1 & \{15,34,234\}  \bullet   \\
 1 & 1 & 1 & 49/50 & 16/25 & 16/25 & 1 & 16/25 & 1 & 1 & \{15,234\}   \\
 1 & 1 & 1 & 49/50 & 4/5 & 16/25 & 1 & 4/5 & 1 & 1 & \{15,24,234\}   \\
 1 & 1 & 1 & 49/50 & 89/100 & 73/100 & 1 & 89/100 & 1 & 1 & \{15,24,234\}   \\
 1 & 1 & 1 & 49/50 & 49/50 & 41/50 & 1 & 49/50 & 1 & 1 & \{15,24,234\}   \\
 1 & 1 & 1 & 1 & 1 & 21/25 & 1 & 1 & 1 & 1 & \{24\}   \\
 21/25 & 1 & 21/25 & 1 & 1 & 21/25 & 1 & 1 & 21/25 & 1 & \{35,124\}   \\
 77/100 & 1 & 21/25 & 1 & 1 & 21/25 & 1 & 1 & 77/100 & 1 & \{12,35,124\}  
\medskip \\
 1 & 1 & 1 & 1 & 91/100 & 3/4 & 91/100 & 91/100 & 91/100 & 91/100 & \{24,2345\}   \\
 1 & 1 & 1 & 1 & 91/100 & 3/4 & 1 & 91/100 & 1 & 1 & \{24,234\}   \\
 1 & 1 & 1 & 1 & 3/4 & 3/4 & 1 & 3/4 & 1 & 1 & \{234\}   \\
 1 & 1 & 1 & 49/50 & 73/100 & 73/100 & 1 & 73/100 & 1 & 1 & \{15,234\}    
\end{matrix}.
$$
}
\end{center}
\caption{\label{table:23ultrametrics}
The $23$ ultrametrics (with tree topologies)
at the nodes in Fig.~\ref{fig:threebyten}.}
\end{table}

Each of the $23$ nodes in our tropical triangle represents an equidistant tree.
In Table \ref{table:23ultrametrics} we list the $23$ ultrametrics,
along with their tree topologies.
Those marked with a bullet $\bullet$ are the rows of $\mathcal{D}$.
The boundary of ${\rm tconv}(\mathcal{D})$
 is given by the first $19$ rows, in counterclockwise order. 
Rows $1$ to $6$ form  the tropical segment from  $D^{(1)}$ to $D^{(2)}$,
rows $7$ to $12$ form  the tropical segment from  $D^{(2)}$ to $D^{(3)}$,
and rows $13$ to $19$ form  the tropical segment from  $D^{(3)}$ to $D^{(1)}$.
The last four rows are the interior nodes, from top to bottom
in Fig.~\ref{fig:threebyten}.
The tropical segment from $D^{(2)}$ to $D^{(3)}$ has depth $1$, but
the other two segments have depth $2$.
See Section~\ref{sec5} for the definition of ``depth''.
Note that some of the breakpoints, such as 
that given in row $6$, lie in the interior of a maximal cone in
the tree space $\,\mathcal{U}_5 = {\rm Trop}(L_5)$.  The red circle in Fig.~\ref{fig:threebyten} is the {\em tropical centroid},
a concept to  be introduced in Section \ref{sec6}.
\end{example}

The polyhedral geometry package  {\tt Polymake} \cite{polymake}
 can compute the tropical  convex hull of a finite set of points.
 It can also visualize such a tropical polytope, with its cell complex structure,
  provided its dimension is 
 $2$ or $3$.  Fig.~\ref{fig:threebyten}    was drawn using {\tt Polymake}.

 Matroid theory furnishes
the appropriate level of generality for tropical convexity.
We refer to the text book reference  \cite[\S 4.2]{MS}
for an introduction. Let $M$ be any matroid of rank $r$ on 
 the ground set $[e] = \{1,2,\ldots,e\}$.
The associated {\em tropical linear space} ${\rm Trop}(M)$ is 
a polyhedral fan of dimension $r-1$ in the quotient space $\RR^e / \RR {\bf 1}$.
For historical reasons, the tropical linear space ${\rm Trop}(M)$
is also known as the {\em Bergman fan} of the matroid $M$; see \cite{Ard,AK}.
The tree space $\mathcal{U}_m$ arises when
$M$ is the graphic matroid \cite[Example 4.2.14]{MS}
associated with the complete graph $K_m$.
Here $[e]$ indexes the  edges of $K_m$, so
$e = \binom{m}{2}$ and the rank is
$r = m-1$.  Fig.~\ref{fig:eins} (left)
shows  a rendition of  the Petersen graph in
\cite[Figure 4.1.4]{MS}.

The tropical linear space ${\rm Trop}(M)$ is tropically convex
\cite[Proposition 5.2.8]{MS}. Hence the notions of
tropical segments, tropical triangles, etc.~defined above extend immediately to
${\rm Trop}(M)$.  This convexity structure on ${\rm Trop}(M)$
is extrinsic and global. It is induced from the ambient space
$\RR^e / \RR {\bf 1}$, so it does not rely on
choosing a subdivision or local coordinates.

We can also define the structure of an orthant space
on ${\rm Trop}(M)$. This requires the choice of
a simplicial fan structure on ${\rm Trop}(M)$.
 Feichtner \cite{Fei} developed a theory of such fan.
Each is determined by a
collection of $n$ flats of $M$, known as a  {\em building set}.
From these one constructs  a simplicial complex
$\Omega$ on $[n]$,
known as a  {\em nested set complex}.
This $\Omega$ may or may not be flag.
The finest fan structure  in this theory
arise when all flats of $M$ are in the collection.
Here the nested set complex $\Omega$ is flag: it is
the order complex of the geometric lattice of $M$.
 See  Exercise 10 in \cite[Chapter 4]{MS}.

\begin{example} \label{ex:uniform} \rm
Let $M$ be the uniform matroid of rank $r$ on $[e]$.
The tropical linear space ${\rm Trop}(M)$ is the set of all
vectors $u$ in   $\RR^e/\RR {\bf 1}$  whose largest coordinate is
attained at least $e-r+1$ times. The proper flats of $M$
are the non-empty subsets of $[e]$ having cardinality at most $r-1$.
Their number is $n = \sum_{i=1}^{r-1} \binom{e}{i}$.
Ordered by inclusion, these form the geometric lattice of $M$.
Its order complex is the first barycentric subdivision
of the $(r-1)$-skeleton of the $(e-1)$-simplex.
This is a flag simplicial complex $\Omega$ with $n$
vertices. The corresponding orthant space
$\mathcal{F}_\Omega$ in $\RR^n$
defines the structure of an orthant space on
${\rm Trop}(M) \subset \RR^e /\RR {\bf 1}$.
\end{example}

While the order complex of the geometric lattice of any matroid $M$
makes {\small ${\rm Trop}(M)$} into an orthant space,
many matroids have smaller building sets
whose nested set complex is flag as well.
Our primary example is ultrametric space 
$\mathcal{U}_m = {\rm Trop}(L_m)$.
The flats of $L_m$ correspond to proper set partitions
of $[m]$. Their number is $n = B_m{-}2$, where 
$B_m$ is the {\em Bell number}. The resulting
orthant space on $\mathcal{U}_m$ is
the {\em $\tau$-space} of Gavruskin and Drummond~\cite{GD}.
The subdivision of $\mathcal{U}_m$ given by the nested sets of clades
is much coarser. It has only $n = 2^m{-}m{-}2$ rays,
and its orthant space  gives the BHV metric.
This  is different from the $\tau$-metric, by \cite[Proposition 2]{GD}. Note that
\cite[Figure 4]{GD} is the same as our Example~\ref{ex:differentomega}.

Each orthant space structure defines a Euclidean metric on ${\rm Trop}(M)$.
These metrics differ dramatically from the tropical metric, to be defined next,
in (\ref{eq:tropmetric}). Euclidean metrics
on ${\rm Trop}(M)$  are {\bf intrinsic} and
do not extend to the ambient space 
$\RR^e/\RR {\bf 1}$. Distances are computed
by identifying ${\rm Trop}(M)$ with the orthant space
$\mathcal{F}_\Omega$ of a 
 nested set complex $\Omega$ that is flag.
 On the other hand, the tropical metric is {\bf extrinsic}.
 It lives on $\RR^e/\RR {\bf 1}$ and is defined on
$ {\rm Trop}(M)$ by restriction.
The tropical distance between two points~is computed as follows:
\begin{equation}
\label{eq:tropmetric} d_{\rm tr}(v,w) \,\, = \,\,
\max \bigl\{\, |v_i - w_i  - v_j + w_j| \,\,:\,\, 1 \leq i < j \leq e \,\bigr\} . 
\end{equation}
This is also known as the
{\em generalized Hilbert projective metric} 
\cite[\S 2.2]{AGNS}, \cite[\S 3.3]{CGQ}.
Unlike in the Euclidean setting of
Theorem \ref{thm:owenprovan},
geodesics in that metric are  not unique.

\begin{proposition}
For any two distinct points $v,w \in \RR^e/\RR {\bf 1}$,
there are many geodesics between $v$ and $w$
in the tropical metric. One of them is the tropical line segment
${\rm tconv}(v,w)$.
\end{proposition}

\begin{proof}
If $u$ is any point whose coordinates lie between
those of $v$ and $w$, then $d_{\rm tr}(v,w) = d_{\rm tr}(v,u) + d_{\rm tr}(u,w)$.
Hence, any path from $v$ to $w$
that is monotone in each coordinate is a geodesic.
One such path is the tropical segment ${\rm tconv}(v,w)$
in the max-plus arithmetic.
\end{proof}

One important link between the tropical metric
and tropical convexity is the nearest-point map,
to be described next. Let $\mathcal{P}$ be any 
tropically convex closed subset of $\RR^e/\RR {\bf 1}$,
and let $u$ be any vector in $\RR^e $.
Let $\mathcal{P}_{\leq u}$ denote the subset of
all points in $\mathcal{P}$ that have a representative
$v \in \RR^e$ with $v \leq u$ in the coordinate-wise order.
In tropical arithmetic this is expressed as $v \oplus u = u$.
If $v$ and $v'$ are elements of $\mathcal{P}_{\leq u}$
then so is their tropical sum $v \oplus v'$.
It follows that $\mathcal{P}_{\leq u}$ contains a unique
coordinate-wise maximal element, denoted
${\rm max}(\mathcal{P}_{\leq u})$.

\begin{theorem} Given any tropically convex closed subset $\mathcal{P}$
of $\,\RR^e/ \RR {\bf 1}$, consider the function
\begin{equation}
\label{eq:nearestpoint}
 \pi_\mathcal{P} \,:\, \RR^e/ \RR {\bf 1}
\rightarrow \mathcal{P}\,, \,\,
u \mapsto {\rm max}(\mathcal{P}_{\leq u}) . 
\end{equation}
Then $\pi_\mathcal{P}(u)$ is the unique point in $\mathcal{P}$
that minimizes the tropical distance to $u$.
\end{theorem}

This result was proved by Cohen, Gaubert and Quadrat
in \cite[Theorem 18]{CGQ}. See also \cite{AGNS}. 
In Section \ref{sec6} we shall discuss the important case when $\mathcal{P} = {\rm Trop}(M)$
is a tropical linear space. The subcase when $\mathcal{P}$ is a tropical hyperplane
appears in \cite[\S 7]{AGNS}.

We close this section by considering a tropical polytope,
{\small ${\rm tconv}(D^{(1)},D^{(2)},\ldots,D^{(s)})$.}
The $D^{(i)}$ are points in $\RR^e/\RR {\bf 1}$. For instance,
they might be ultrametrics in $ \mathcal{U}_m$.  Then
$$ \pi_\mathcal{P} (D) \,= \,
\lambda_1 \odot  D^{(1)} \,\oplus \,
\lambda_2 \odot  D^{(2)} \,\oplus \, \cdots \,\oplus \,
\lambda_s \odot  D^{(s)}  ,
\quad {\rm where} \,\, \lambda_k = {\rm min}(D-D^{(k)}) .
$$
This formula appears in \cite[(5.2.3)]{MS}. It allows us to easily
project an ultrametric $D$ (or any other point in $\RR^e$) onto the
tropical convex hull of $s$ given ultrametrics.

\section{Experiments with Depth}
\label{sec5}

It is natural to compare tropical convexity with geodesic convexity. 
One starts by comparing the line segments ${\rm gconv}(v,w)$ and ${\rm tconv}(v,w)$,
where $v,w $ are points in a tropical linear space ${\rm Trop}(M) $.
Our first observation is that tropical segments generally do not
obey the combinatorial structure imposed by ordered partitions 
$(\mathcal{A},\mathcal{B})$. For instance, if $v$ and $w$ represent equidistant trees 
then every clade used by a tree in 
the geodesic segment ${\rm gconv}(v,w)$ must be a clade of $v$ or $w$.
This need not hold for trees in the tropical segment ${\rm tconv}(v,w)$.

\begin{example} \rm 
Let $v = D^{(1)}$ and $w = D^{(2)}$  in Example  \ref{ex:TropTriangle}.
Their tree topologies are given by the sets of clades
$\{12,35,124\}$ and $\{23,45, 2345\}$.
Consider the breakpoints of the tropical segment ${\rm tconv}(v,w)$ 
given in lines 3 and 4 of Table \ref{table:23ultrametrics}.
Both trees have the new clade $24$.
\end{example}

This example suggests that tropical segments might be worse
than Euclidean geodesics. However, as we shall now argue,
the opposite is the case: for us, tropical segments are better.
 We propose the following
quality measure for a path $P$ in an orthant space $\mathcal{F}_\Omega$.
Suppose that $\mathcal{F}_\Omega$ has
 dimension $d$. Each point of $P$ lies in the relative interior of a unique
orthant $\mathcal{O}_\sigma$, and we say that this point has {\em codimension}
$d-{\rm dim}(\mathcal{O}_\sigma)$. We define the {\em depth}
of a path $P$ as the maximal codimension of any point in $P$. For 
instance, the depth of the Euclidean geodesic in Theorem \ref{thm:owenprovan}
is the maximum of the numbers
$\sum\limits_{k=1}^i |A_k| -  \sum\limits_{k=1}^{i-1} |B_k|$,
where $i$ runs over $\{1,2,\ldots,q\}$.
These are the codimensions of the breakpoints of ${\rm gconv}(v,w)$.

  Geodesics of small depth are desirable.
A  cone path has depth~$d$.
Cone paths are bad from a statistical perspective
because they give rise to {\em sticky means}, see e.g.~\cite{openbook},
\cite{Mil} or \cite[\S 5.3]{MOP}.
Optimal geodesics have depth $0$. 
Such geodesics are line segments within a single orthant.
These occur  if and only if the starting point and target point are in the same orthant. 
If the two given points are not in the same orthant then
the best-case scenario is depth $1$, which means that
each transition is through an orthant of codimension~$1$ 
in $\mathcal{F}_\Omega$.
We conducted two experiments, to assess the
depths of ${\rm gconv}(v,w)$ and ${\rm tconv}(v,w)$.

\begin{experiment}[Euclidean Geodesics] \label{exp:eins} \rm
For each $m \in \{4,5,\ldots,20\}$,
we sampled $1000$ pairs 
$\{v,w\}$ from the compact tree space $\mathcal{U}_m^{[1]}$,
and for each pair we computed the depth of ${\rm gconv}(v,w)$.
  The sampling scheme is described below. The depths are integers between $0$ and $m-2$.
\begin{algorithm2}[Sampling normalized equidistant
  trees with $m$ leaves]\label{al1} \label{alg:generating} 
\\
{\bf Input}:  The number $m$ of leaves, and the sample size $s$. \\
 {\bf Output}: A sample of $s$ random equidistant trees
in the compact tree space $\,\mathcal{U}_m^{[1]}$.
\begin{enumerate} 
\item Set $\mathcal{S} = \emptyset$. 
\item For $i = 1, \ldots s$, do  
\begin{enumerate}
\item Generate a tree $D_i$ using the function {\tt rcoal} from
the {\tt  ape} package {\rm \cite{APE}} in {\tt R}. 
\item Randomly permute the leaf labels on the metric tree $D_i$.  
\item Change the clade nested structure of $D_i$ by randomly applying the nearest
  neighbor interchange (NNI) operation $m$ times.  
\item  $\!\!$Turn $D_i$ into an equidistant tree using the {\tt ape} function {\tt
    compute.brtime}.
    \item Normalize $U_i$ so that
the distance from the root to each leaf is $\frac{1}{2}$.  
\item Add $D_i$ to the output set $\mathcal{S}$.  
\end{enumerate}  
\item Return $\mathcal{S}$.
\end{enumerate}
\end{algorithm2}

Table \ref{table4} shows the distribution of the depths.
For instance, the first row concerns $1000$ random
geodesics on the $2$-dimensional polyhedral fan $\mathcal{U}_4$ depicted in Fig.~\ref{fig:eins}.
 Of these geodesics, $8.4\%$ were in a single orthant,
$58.4\%$ had depth $1$, and $33.2 \%$ were cone paths.
For $m=20$, the fraction of cone paths was $6.2 \%$.
The data in Table~\ref{table4} are based on the
sampling scheme in Algorithm \ref{al1}. 
\setlength{\tabcolsep}{4.2pt}
\begin{table}[h!]
\begin{center}
\tiny
\begin{tabular}
{|c|c|c|c|c|c|c|c|c|c|c|c|c|c|c|c|c|c|c|c|}
\hline
$\!\!m \backslash$depth \!\!\!\!\!&0&1&2&3&4&5&6&7&8&9&10&11&12&13&14&15&16&\!\!17\!\!&\!\!18\!\!\\
\hline
4&8.4&\!\!\!\!\! 58.4 \!\!\!\!\! &\!\!\!\! 33.2 \!\!\!\! &&&&&&&&&&&&&&&&\\
\hline
5&1.6&\!\!\!\!\! 26.4 \!\!\!\!\!& \!\!\!\! 47.4 \!\!\!\!&\!\!\!\! 24.6 \!\!\!\! &&&&&&&&&&&&&&&\\
\hline
6&0.2&\!\!\!\! \!13.2 \!\!\!\!\!&\!\!\!\! 36.7 \!\!\!\! & \!\!\!\! 31.5 \!\!\!\! &\!\!\!\! 18.4 \!\!\!\! &&&&&&&&&&&&&&\\
\hline
7&0&4&\!\!\!\! \!25.9\! \!\!\!\! & \!\!\!\! 29.9 \!\!\!\! & \!\!\!\! 22.2 \!\!\!\! &18&&&&&&&&&&&&&\\
\hline
8&0&1.1&15&\!\!28.9\!\!&25&\!\!17.1\!\!&\!\!12.9\!\!&&&&&&&&&&&&\\
\hline
9&0&0.8&8&\!\!22.1\!\!&\!\!25.9\!\!&\!\!18.3\!\!&\!\!14.5\!\!&\!\!10.4\!\!&&&&&&&&&&&\\
\hline
10&0&0.4&3.3&\!\!17.2\!\!&\!\!22.3\!\!&\!\!20.6\!\!&\!\!14.1\!\!&\!\!13.2\!\!&8.9&&&&&&&&&&\\
\hline
11&0&0.2&1.5&\!\!10.4\!\!&\!\!17.6\!\!&\!\!20.3\!\!&\!\!16.8\!\!&\!\!12.8\!\!&\!\!11.1\!\!&9.3&&&&&&&&&\\
\hline
12&0&0.2&0.1&6&\!\!14.1\!\!&\!\!20.4\!\!&\!\!13.9\!\!&\!\!14.6\!\!&\!\!12.7\!\!&\!\!10.5\!\!&7.5&&&&&&&&\\
\hline
13&0&0.2&0.4&4.2&10.1&17.2&15.9&12.5&11&9.8&9.1&9.6&&&&&&&\\
\hline
14&0&0.2&0&2.7&9.3&14.9&15.5&12.2&11.3&\!\!10.4\!\!&8.7&8&6.8&&&&&&\\
\hline
15&0&0.1&0&1.4&5.9&12.7&13&13.1&11.3&9.2&8.9&8.5&7.5&8.4&&&&&\\
\hline
16&0&0&0&1&5&11.2&11.4&11.3&11.2&9.9&8.1&9.1&7.3&6.7&7.8&&&&\\
\hline
17&0&0&0&0.2&3.4&5.9&10.7&11&11.2&\!\!11.5 \!\!&8.4&7.9&7.9&6.2&8.5&7.2&&&\\
\hline
18&0&0&0.1&0.4&1.5&6.5&8.7&10.5&10.9&9.7&7.9&7.5&7.1&8.7&7.7&6.5&6.3&&\\
\hline
19&0&0&0&0.2&1.6&5&7.2&9.3&9.6&8.5&7.5&8.3&7.4&6.1&9.2&7.4&6.8&\!\!\!5.9\!\!\!&  \\
\hline
20&0&0&0&0&0.5&3&6.7&7.6&11.2&9.8&9.4&8.2&5.9&7.5&6.9&6.9&4.5&\!\!\!5.7\!\!\!&
\!\!\!\!\!6.2\!\!\!\!\! \\
\hline
\end{tabular}
\end{center}
\caption{\label{table4}
The rows are labeled by the number $m$ of taxa and the columns
are labeled by the possible depths of a geodesic in tree space.
The entries in each row sum to $100\%$. They are the
frequencies of the depths among $1000 $ geodesics,
randomly sampled using Algorithm \ref{al1}.
}
\end{table}
\end{experiment}

Next we perform our experiment with the tropical line segments ${\rm  tconv}(v,w)$.

\begin{experiment}[Tropical Segments]  \rm
For each $m \in \{4,5,\ldots,20\}$, we  revisited the
same $1000$ pairs $\{v,w\}$ 
from Experiment \ref{exp:eins}, and
we computed their tropical segments
${\rm tconv}(v,w)$.
Table \ref{table5} shows the distribution of their depths.
\begin{table}[h]
\begin{center}
{\tiny
\begin{tabular}
{|c|c|c|c|c|c|c|c|c|c|c|c|}
\hline
$ m \backslash$depth\!\!&0&1&2&3&4&5&6&7&8&9&10 \\
\hline
4 & 8.1 & 88.7 &   3.2  &  &&&&&&&\\ \hline
5 & 1.5 & 84.7 &  13.8 & 0 &&&&&&&\\ \hline
6 & 0.3 & 69.9 &  29.8  &  0 & 0 &&&&&&\\ \hline
7  &   0 & 55.7 &  44.1 & 0.2 & 0  & 0 &&&&&\\ \hline
8  &   0 & 42.8 &  56.9 & 0.2 & 0.1 & 0 & 0 &&&&\\ \hline
9  &   0 & 28.0 &  71.4 & 0.6 &  0 & 0 & 0 & 0 &  &&\\ \hline
10  & 0 & 20.7 &  78.2 & 1.1 & 0  & 0 & 0 & 0 & 0 &&\\ \hline
11  & 0 & 10.8 &  88.0 & 1.1 & 0.1 & 0 & 0 & 0 & 0 & 0 &\\ \hline
12  & 0 &  7.8 &   89.5 & 2.4 & 0.2 & 0.1 & 0 & 0 & 0 & 0 & 0 \\ \hline
13  & 0 &  5.3 &  90.8 & 3.5  & 0.3 & 0.1 & 0 & 0 & 0 & 0 & 0 \\ \hline
14  & 0 &  2.5 &  92.1 & 4.9 & 0.5 & 0 & 0 & 0 & 0 & 0 & 0 \\ \hline
15  & 0 &  1.9 &  90.4 & 6.8 & 0.9 & 0 & 0 & 0 & 0 & 0 & 0 \\ \hline
16  & 0 &  0.4 &  88.7 & 9.4 & 1.1 & 0.4 & 0 & 0 & 0 & 0 & 0 \\ \hline
17  & 0 &  0.8 &  87.5 & 9.4 & 2.3 & 0 & 0 & 0 & 0 &0 & 0 \\ \hline
18  & 0 &  0.8 &  86.1 & 9.7 & 3.1 & 0.2 & 0 & 0.1 & 0 & 0 & 0 \\ \hline
19  & 0 &  0.3 &  84.1 & 11.3 & 3.4 & 0.8 & 0.1 & 0 & 0 & 0 & 0 \\ \hline
20  & 0 &  0.1 &  78.4 & 16.5 & 3.9 & 1.0 & 0 & 0.1 & 0 & 0 & 0 \\
\hline
\end{tabular}
}
\end{center}
\caption{\label{table5}
Frequencies of the depths among $1000 $ tropical segments.
The same input data as in Table~\ref{table4} was used.
}
\end{table}
There is a dramatic difference between Tables \ref{table4} and \ref{table5}.
The depths of the Euclidean geodesics are much larger
than those of the tropical segments.
Since small depth is desirable, this suggests that
the tropical convexity structure may have good statistical properties.
\end{experiment}

$\!\!\!\!$Triangles show even more striking differences: While tropical triangles 
 ${\rm tconv}(a,b,c)$ are  $2$-dimensional,
 geodesic triangles ${\rm gconv}(a,b,c)$
can have arbitrarily high dimension, by Theorem  \ref{thm:highdim}.
In spite of these dimensional differences,
$\,  {\rm tconv}(a,b,c) \, $ is usually not contained in
$  \, {\rm gconv}(a,b,c)$. In particular, this is
the case in the following example.

\begin{example} \label{ex:bookex} \rm
Fix $e=4, r=3$ and
let $M$ be the matroid with bases
$124, 134,234$.  The tropical plane ${\rm Trop}(M)$
is defined in $\RR^4/\RR {\bf 1}$ by 
 $x_1 \oplus x_2 \oplus x_3$.
Geometrically, this is the open book with three 
pages in Example~\ref{ex:threepages}.
The following points lie in ${\rm Trop}(M)$:
$$ a = (0,1,1,2), \quad b =   (1,0,1,4) , \quad c =  (1,1,0,6) . $$
Fig.~\ref{fig:beauty} shows the two triangles
${\rm gconv}(a,b,c)$  and ${\rm tconv}(a,b,c)$ inside that open book.
See Example~\ref{ex:threepages} and
Fig.~\ref{fig:subtle} for the derivation of the geodesic triangle.
\end{example}


\begin{figure}[h!]
  \begin{center}
    \includegraphics[width=5.5cm]{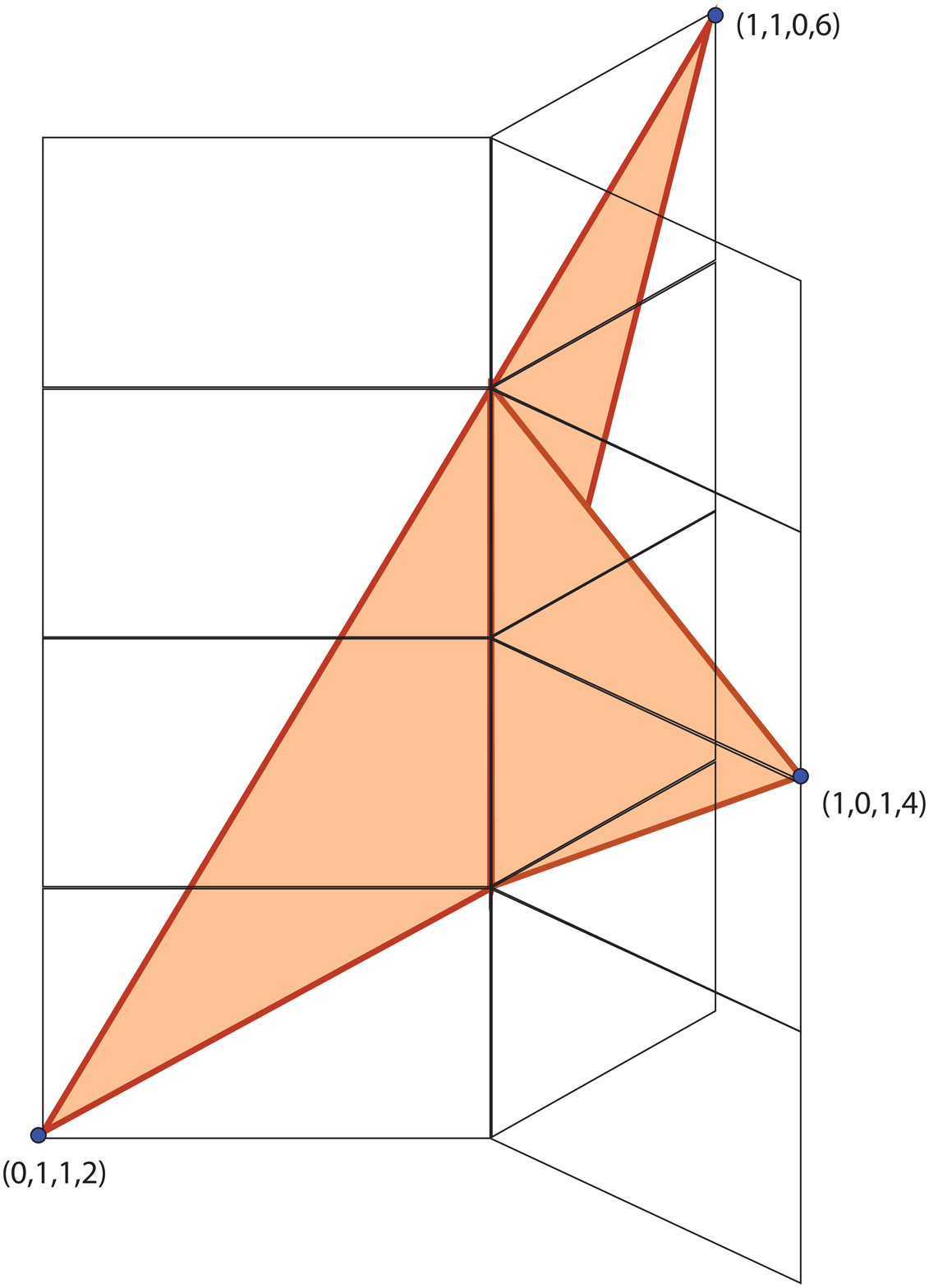} \,\,\,
        \includegraphics[width=5.5cm]{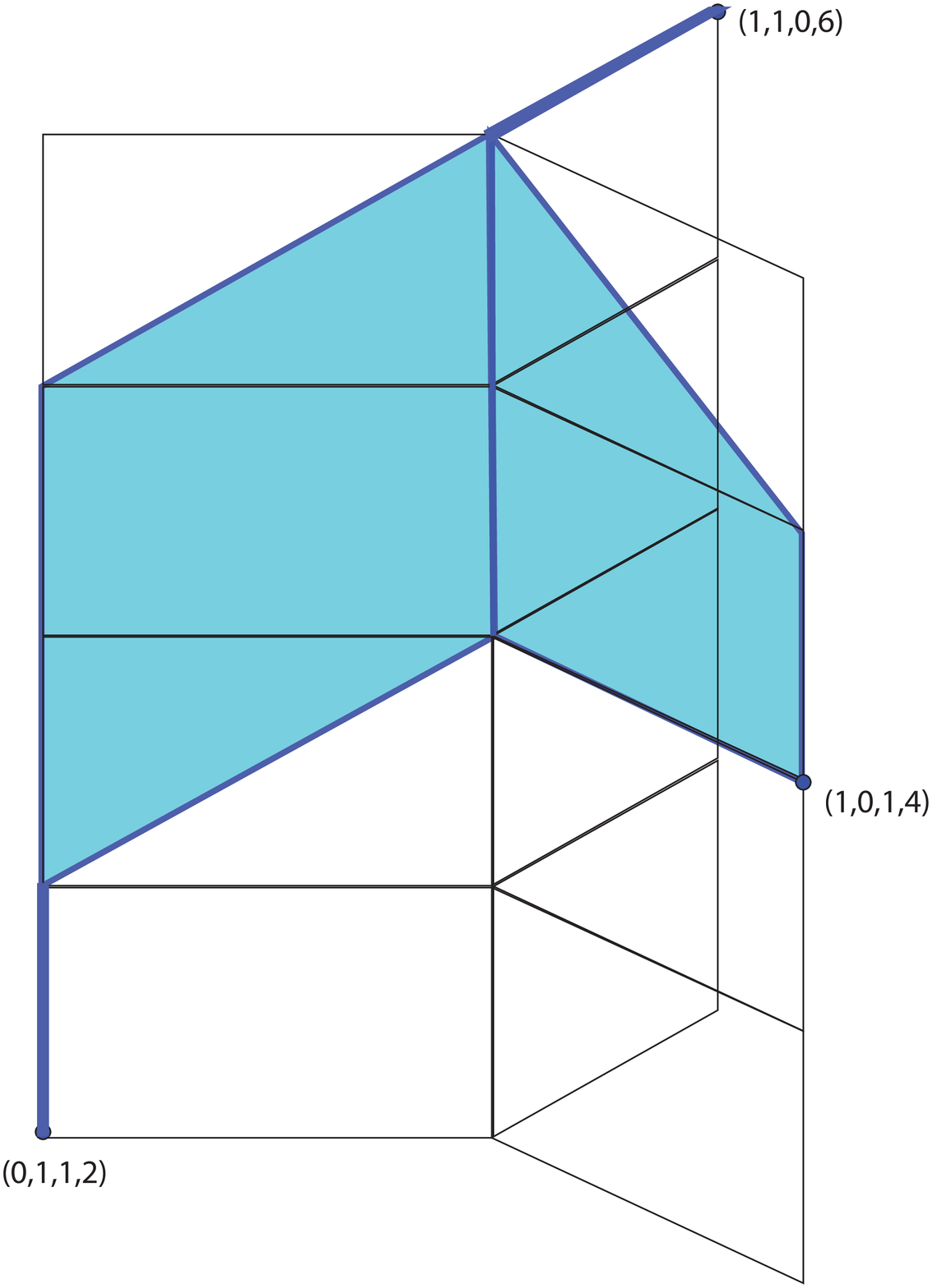}
  \end{center}
  \caption{\label{fig:beauty} 
Two convex hulls of three points $a,b,c$  inside the open book with three pages.
  The geodesic triangle is shown on the left, while the
  tropical triangle is on the right.
    } \end{figure}
  
\section{Towards Tropical Statistics}
\label{sec6}

Geometric statistics is concerned with the analysis of
data sampled from highly non-Euclidean spaces \cite{openbook, Mil}.
The section title above is meant to suggest
the possibility that objects and methods from
tropical geometry can play a role in this development.
As an illustration consider the 
widely used technique of Principal Component Analysis (PCA).
This serves to reduce the dimension of
high-dimensional data sets, by projecting these
onto lower-dimensional subspaces in the data space.
The geometry of dimension reduction is essential
in phylogenomics, where it can provide insight into relationships and
evolutionary patterns of a diversity of organisms, from humans, plants
and animals, to microbes and viruses.   

To see how tropical convexity might come in,
consider the work of Nye \cite{Nye}
in statistical phylogenetics. 
Nye developed PCA for BHV tree space.
We identify tree space with $\mathcal{U}_m$
and we  sketch the basic ideas.
Nye defined a line $L$ in $\mathcal{U}_m$
to be an unbounded path
such that every bounded subpath is a BHV geodesic.
Suppose that $L$ is such a line,
and $u \in \mathcal{U}_m \backslash L$.
Proposition 2.1 in \cite{Nye} shows that $L$ contains
a unique point $x$ that is closest to $u$
in the BHV metric. We call $x$ the {\em projection}
of $u$ onto the line $L$.
Given $u$ and $L$, we can compute $x$ as follows.
Fixing a base point $L(0)$ on the line, one choses
a geodesic parametrization  $L(t)$  of the line.
This means that $t$ is the distance $d(L(0),L(t))$.
Also let $k$ denote the distance from $u$ to $L(0)$.
By the triangle inequality, the desired point is
$x = L(t^*)$ for some $t^* \in [-k,k]$.
The distance $d(x,L(t))$ is a continuous
function of $t$. Our task is to find the value of $t^*$ which
minimizes that function on the closed interval  $[-k,k]$. This is done
 using numerical methods.
The uniqueness of $t^*$ follows from the ${\rm CAT}(0)$ property.

Suppose we are given a collection   $\{v^1, v^2,\ldots, v^s\} $ of tree metrics on $m$ taxa.
This is our data set for phylogenetic analysis. Nye's method computes a first
principal line (regression line) for these data inside BHV space. This is done
as follows. One first computes the {\em centroid} $x^0$ of the $s$ given trees.
This can be done using the iterative method in \cite[Theorem 4.1]{BHV}.
Now, the desired regression line $L$ is one of geodesics through $x^0$.
For any such line $L$, we can compute the projections 
$x^1,\ldots,x^s$ of the data points $v^1,\ldots,v^s$.
The goal is to find the line $L$ that minimizes (or, maximizes)
a certain objective function. Nye proposes two such functions:
$$ f_\perp(L):= \sum_{i=1}^s d(v^i,x^i)^2 \quad \hbox{or} \quad
 f_\para(L):= \sum_{i = 1}^s d(x^0,x^i)^2. $$
The first function of $L$ above can be minimized and the second function above can be maximized using an iterative numerical procedure.

While the paper \cite{Nye} represents a milestone concerning
statistical inference in BHV tree space, it left the open problem of
computing higher-dimensional principal components.
First, what are geodesic planes?
Which of them is the regression plane for $v^1,v^2,\ldots,v^s$?
Ideally, a plane in tree space would be a $2$-dimensional complex
that contains the geodesic triangle formed by any three of its points.
Outside a single orthant, do such planes even exist?
Such questions were raised in \cite[\S 6]{Nye}.
Our Theorem~\ref{thm:highdim} suggests that the
answer is negative.

On the other hand, tropical convexity and tropical linear algebra in $\mathcal{U}_m$
 behave better. Indeed,  each triangle ${\rm tconv}(v^1,v^2,v^3)$
 in $\RR^e \! / \RR {\bf 1}$ is $2$-dimensional and spans a tropical plane;  each  
tropical tetrahedron  ${\rm tconv}(v^1,v^2,v^3,v^4)$
is $3$-dimensional  and spans a tropical $3$-space.
These tropical linear spaces are best represented by their 
Pl\"ucker coordinates; cf.~\cite[\S 4.4]{MS}.

In what follows we take first steps towards the introduction of
tropical methods into geometric statistics.
We study  {\em tropical centroids} and  {\em projections}
onto tropical linear spaces.

In any metric space $\mathcal{M}$, one can study
 two types of ``centroids": one is the Fr\'echet mean and the
other is the Fermat-Weber point.  
Given a finite sample $\{v^1,v^2,\ldots,v^s\}$ of points in $\mathcal{M}$,
a {\em Fr\'echet mean}  minimizes the
sum of squared distances to the points.
A {\em Fermat-Weber point} $y$ minimizes the
sum of distances to the given points. 
\begin{equation}\label{fermat}
y \,:= \, \argmin_{z \in \mathcal{M}} \sum_{i = 1}^s d(z, v^i).
\end{equation}
Here we do not take the square.
We note that $y$ is generally not a unique point but refers to the set of all minimizers.
Millar et.al.~\cite{MOP} took the Fr\'echet mean
for the centroid in orthant spaces.
In the non-Euclidean context of tropical geometry, we prefer
to work with the Fermat-Weber points.
If $ v^1,\ldots, v^s \in \mathcal{M} \subset \RR^e / \RR {\bf 1}$
 then a {\em tropical centroid} is any solution $y$ to (\ref{fermat}) with  $d = d_{\rm tr}$.
In unconstrained cases, we can use the
following linear program to compute tropical centroids:

\begin{proposition}
Suppose $\mathcal{M} = \RR^e / \RR {\bf 1}$.
Then the set of tropical centroids is a convex polytope.
It consists of all optimal solutions $y = (y_1,\ldots,y_e)$
to the following linear program:
\begin{equation}
\label{eq:ourLP}
 \begin{matrix}
{\rm minimize} \quad d_1 + d_2 + \cdots + d_s \quad \hbox{\rm subject to} \qquad \qquad \qquad \\
\qquad \qquad 
y_j-y_k- v^i_j + v^i_k  \geq -d_i  \quad \hbox{for all} \,\,\, i = 1, \ldots ,s \,\,\hbox{and}\,\,\, 1\le j,k\le e,\\
\qquad \qquad 
y_j-y_k- v^i_j + v^i_k  \leq \phantom{-}d_i 
 \quad \hbox{for all} \,\,\, i = 1, \ldots ,s \,\,\hbox{and} \,\,\, 1\le j,k\le e.
\end{matrix}
\end{equation}
\end{proposition}

We refer to the subsequent paper \cite{LY} for details, proofs, and further results on the topic
discussed here.
If $\mathcal{M}$ is a proper subset of $\RR^e/\RR {\bf 1}$
then the computation of tropical centroids is highly dependent
on the representation of $\mathcal{M}$.
For tropical linear spaces, $\mathcal{M} = {\rm Trop}(M)$,
we must solve a linear program  on each maximal cone.
The question remains how to do this efficiently.

\begin{example}\label{ex:centroid2} \rm
Consider the rows of the $3 \times 10$-matrix in Example \ref{ex:threebyten}.
We compute the tropical centroid of these three points in
 $\mathcal{M}  = \mathcal{U}_5
\subset \RR^{10}/\RR {\bf 1}$. To do this, we first compute the set of
 all tropical centroids in $\RR^{10} / \RR {\bf 1}$.
This is a $6$-dimensional classical polytope, consisting of all
optimal solutions to (\ref{eq:ourLP}). The intersection of that polytope
with tree space $\mathcal{U}_5$ equals the parallelogram
$$
 \begin{matrix}
\bigl\{
\left(1, 1, 1, 1,  \frac{61}{100} {+} y, 
\frac{61}{100}{+}x{+}y, \frac{3}{4}+y,
\frac{61}{100}{+}y, \frac{3}{4}{+}y, \frac{3}{4}{+}y\right)
:  0 \leq x \leq \frac{43}{100},\, 0 \leq y \leq \frac{7}{50} \bigr\}.
\end{matrix}
$$
This is mapped to a single (red) point inside the tropical triangle
in Fig.~\ref{fig:threebyten}.
\end{example}

Example \ref{ex:centroid2} shows that tropical centroids
of a finite set of  points generally do not lie in the tropical convex hull
of those points. For instance, the tropical centroid of $\{D^{(1)},D^{(2)},D^{(3)}\}$
that is obtained by setting $x=y=0$ in the parallelogram above
does not lie in ${\rm tconv}(D^{(1)},D^{(2)}, D^{(3)} )$.

We now come to our second and last topic in this section, namely
projecting onto subspaces.
Let $L_{\bf w}$ be a {\em tropical linear space} of dimension $r-1$
in $\RR^e/\RR {\bf 1}$. This concept is to be understood in the inclusive sense
of \cite[Definition 4.4.3]{MS}. The notation $L_{\bf w}$ also comes from \cite{MS}.
Hence ${\bf w} = (w_\sigma)$ is a vector in $\RR^{\binom{e}{r}}$ that lies in the 
{\em Dressian}  $\,{\rm Dr}(r,e)$, as in  \cite[Definition 4.4.1]{MS}. The 
Pl\"ucker coordinates $w_\sigma$  are indexed by subsets $\sigma \in \binom{[e]}{r}$.
Among the $L_{\bf w}$ are the tropicalized linear spaces
\cite[Theorem 4.3.17]{MS}. Even more special are 
linear spaces spanned by $r$  points; cf.~\cite{FR}. If $L_{\bf w}$ is spanned by
  $ x^1,  \ldots, x^r$ in $\RR^e/\RR {\bf 1}$ then its Pl\"ucker coordinate
$w_\sigma$  is the
{\em tropical determinant} of the  $r \times r$-submatrix  indexed by $\sigma$ of the
$r \times e$-matrix  $X = (x^1,\ldots,x^e)$. Note that
all tropical linear spaces $L_{\bf w}$ are tropically convex.

We are interested in the nearest point map
$\pi_{L_{\bf w}}$ that takes a point $u$ to 
the largest point in $L_{\bf w}$ dominated by $u$, as 
seen   in (\ref{eq:nearestpoint}).
From \cite[Theorem 15]{JSY} we have:

\begin{theorem}[The Blue Rule]
The $i$-th coordinate of the point in $L_{\bf w}$ nearest to $u$ is equal~to
\begin{equation}
\label{eq:bluerule}
 \quad \pi_{L_{\bf w}}(u)_i \,\, = \,\,
 {\rm max}_\tau \,{\rm min}_{j \not\in \tau} \bigl( u_j + w_{\tau \cup i} - w_{\tau \cup j} \bigr)
 \qquad {\rm for} \,\,\, i = 1,2,\ldots, e. 
 \end{equation}
Here $\tau$ runs over all $(r-1)$-subsets of $[e]$ that do not contain $i$.
\end{theorem}

The special case of this theorem when $L_{\bf w}$ has the form 
${\rm Trop}(M)$, for some rank $r$ matroid $M$ on $[e]$, 
was proved by Ardila in \cite[Theorem 1]{Ard}.
Matroids correspond to the case when each tropical Pl\"ucker coordinate $w_\sigma$ is
either $0$ or $-\infty$. The application that motivated Ardila's study
was the ultrametric tree space $\mathcal{U}_m$.
Here the nearest-point map computes the largest ultrametric
dominated by a given dissimilarity map, a problem of importance
in phylogenetics. An efficient algorithm for this problem was
given by Chepoi and Fichet \cite{CF}. This was recently
revisited by Apostolico {\it et al.}~in~\cite{ACDP}.

\begin{example} \rm
Following Example \ref{ex:uniform},
let $M$ be the uniform matroid of rank $r$ on $[e]$.
Then ${\rm Trop}(M) = L_{\bf w}$ where ${\bf w}$
is the all-zero vector, i.e.~$w_\sigma = 0$ for $\sigma \in \binom{[e]}{r}$.
By (\ref{eq:bluerule}), the $i$-th coordinate of the nearest point
$ \pi_{L_{\bf w}}(u)$ equals
$\, {\rm max}_\tau \,{\rm min}_{j \not\in \tau} u_j  $.
That nearest point is obtained from $u$ by replacing the $e-r$ largest coordinates
in $u$ by the $r$-th smallest coordinate.
\end{example}

Returning to ideas for geometric statistics,
the Blue Rule may serve as a subroutine for
the numerical computation of regression planes.
Let $u^1,\ldots,u^s$ be data points in $\RR^{e}/\RR {\bf 1}$,
lying in a tropically convex 
subset $\mathcal{P}$ of interest, such as
$\mathcal{P} = \mathcal{U}_m$.
The tropical regression plane of dimension $r-1$ is a
solution to the optimization problem
\begin{equation}
\label{eq:regression}
\argmin_{L_{\bf w}} \sum_{i=1}^s  d_{\rm tr} (u^i, L_{\bf w}).
\end{equation}
Here ${\bf w}$ runs over all points in the Dressian ${\rm Dr}(r,e)$,
or in the tropical Grassmannian ${\rm Gr}(r,e)$.
One might restrict to {\em Stiefel tropical linear spaces} \cite{FR},
i.e.~those that are spanned by points.
Even the smallest case $r=2$ is of considerable interest,
as seen in the study of Nye \cite{Nye}. In his 
approach, we would first compute the tropical centroid inside
$\mathcal{P}$ of the sample $\{u^1,u^2,\ldots,u^s\}$. Fix $x^1$ to be that centroid.
Now $x^2 \in \mathcal{P}$ is the remaining decision variable, and 
we optimize over all tropical lines spanned by
$x^1$ and $x^2$ inside $\RR^e/\RR {\bf 1}$. 
Such a line is a tree with $e$ unbounded rays. 
If the ambient tropically convex set $\mathcal{P} $ is a tropical linear space, such as
our tree space $\mathcal{U}_m$, then
the regression tree $L_{\bf w}$ will always be contained inside $\mathcal{P}$.

 \bigskip 

\begin{small}

{\bf Acknowledgements.} We thank  Simon Hampe, Andrew Francis and Megan Owen 
for helpful conversations.
This project started in the summer of 2015, when all authors 
were hosted by the National Institute for Mathematical Sciences,
Daejeon, Korea. 
Ruriko Yoshida was supported by travel funds from the
 Department of Statistics in the College of Arts and Sciences at the
 University of Kentucky. Bernd Sturmfels thanks the
  US National Science Foundation (DMS-1419018)
  and the Einstein Foundation Berlin.


\end{small}


\begin{thebibliography}{10}


\bibitem{AGNS}
M.~Akian, S.~Gaubert, N.~Viorel and I.~Singer,
{\em Best approximation in max-plus semimodules},
Linear Algebra Appl.~{\bf 435} (2011), pp. 3261--3296. 

\bibitem{ACDP}
A.~Apostolico, M.~Comin, A.~Dress and L.~Parida,
{\em Ultrametric networks: a new tool for phylogenetic analysis},
Algorithms for Molecular Biology {\bf 8} (2013), pp. 7.

\bibitem{Ard}
F.~Ardila, {\em Subdominant matroid ultrametrics},
{Annals~of  Combinatorics} {\bf 8} (2004), pp. 379--389.

\bibitem{ABY}
F.~Ardila, T.~Baker and R.~Yatchak,
{\em Moving robots efficiently using the combinatorics of CAT(0) cubical complexes},
 SIAM J.~Discrete Math.~{\bf 28} (2014), pp. 986--1007.
 
\bibitem{AK}
F.~Ardila and C.~Klivans,
 {\em The Bergman complex of a matroid and phylogenetic trees},
  Journal of Combinatorial Theory Ser.~B {\bf 96} (2006), pp. 38--49.
 
 \bibitem{Ber} M.~Berger,
{\em A Panoramic View of Riemannian Geometry},
Springer Verlag, Berlin, 2003.
 
\bibitem{BHV} L.~Billera, S.~Holmes and K.~Vogtman,
{\em Geometry of the space of phylogenetic trees},
 Advances~in Applied~Mathematics {\bf 27} (2001), pp. 733--767.

\bibitem{Bow} B.~Bowditch, {\em Some results on the geometry of convex hulls
in manifolds of pinched negative curvature},
Comment.~Math.~Helvetici {\bf 69} (1994), pp. 49--81.

\bibitem{CF} V.~Chepoi and B.~Fichet,
{\em $\ell_\infty$ approximation via subdominants},
{Journal of Mathematical Psychology} {\bf 44} (2000), pp. 600--616.
 
\bibitem{CGQ} G.~Cohen, S.~Gaubert and J.P.~Quadrat,
{\em Duality and separation theorems in idempotent semimodules},
Linear Algebra Appl.~{\bf 379} (2004), pp. 395--422.

\bibitem{Fei} E.~Feichtner,
 {\em Complexes of trees and nested set complexes},
  Pacific Journal of Mathematics {\bf 227} (2006), pp. 271--286.
   
\bibitem{FR} A.~Fink and F.~Rinc\'on,
{\em Stiefel tropical linear spaces},
J.~Combin.~Theory~A {\bf 135} (2015), pp. 291--331.

\bibitem{FMPV}   P.T.~Fletcher, J.~Moeller, J.M.~Phillips and S.~Venkatasubramanian,
{\em Computing hulls, centerpoints and VC dimension in positive definite spaces},
presented at Algorithms and Data Structures Symposium, New York, 2011;
original version at {\tt arXiv:0912.1580}.

 \bibitem{GD} A.~Gavruskin and A.~Drummond,
 {\em The space of ultrametric phylogenetic trees}, 
Journal of Theoretical Biology {\bf 403} (2016), pp. 197--208.

\bibitem{polymake} E.~Gawrilow and M.~Joswig, {\em polymake: a
  framework for analyzing convex polytopes}, in Polytopes: combinatorics
  and computation, 43--73, DMV Seminar 29,
  Birkh\"auser, Basel, 2000. 

\bibitem{openbook} T.~Hotz, S.~Huckemann, H.~Le,
J.~Marron, J.~Mattingly, E.~Miller, J.~Nolen, M.~Owen,
V.~Patrangenaru and S.~Skwerer,
{\em Sticky central limit theorems on open books},
Annals~of Applied Probability {\bf 6} (2013), pp. 2238--2258.

\bibitem{JSY}
M.~Joswig, B.~Sturmfels and J.~Yu,
{\em Affine buildings and tropical convexity},
 Albanian J.~Math. {\bf 1} (2007), pp. 187--211.
 
\bibitem{LY} B.~Lin and R.~Yoshida,
 Tropical Fermat-Weber points, {\tt arXiv:1604.04674}.
 
\bibitem{MS}  D.~Maclagan and B.~Sturmfels,
{\em  Introduction to Tropical Geometry},
Graduate Studies in Mathematics, 161, American Mathematical Society, 
Providence, RI, 2015. 

\bibitem{Mil} E.~Miller, {\em Fruit flies and moduli: Interactions between biology
and mathematics}, Notices of the American Mathematical Society
{\bf 62}(10) (2015), pp. 1178--1183.

\bibitem{MOP} E.~Miller, M.~Owen and S.~Provan,
{\em Polyhedral computational geometry for averaging
metric phylogenetic trees}, 
Advances in Applied Mathematics {\bf 68} (2015) pp. 51--91.

\bibitem{Nye} T.~Nye,
{\em Principal components analysis in the space of phylogenetic
  trees}, Annals~of~Statistics~ {\bf 39} (2011), pp. 2716--2739.

\bibitem{OP} M.~Owen and S.~Provan,
{\em  A fast algorithm for computing geodesic distances in tree space},
IEEE/ACM Trans.~Computational Biology and Bioinformatics {\bf 8} (2011), pp. 2--13.

\bibitem{PS} L.~Pachter and B.~Sturmfels,
{\em Algebraic Statistics for Computational Biology}, 
Cambridge University Press, 2005. 

\bibitem{APE} E.~Paradis, J.~Claude and K.~Strimmer, {\em A{PE}:
    analyses of phylogenetics and evolution in  {R} language},
  Bioinformatics {\bf 20} (2004), pp. 289--290.

\bibitem{SS} C.~Semple and M.~Steel,
 {\em Phylogenetics},
  Oxford Lecture Series in Mathematics and its Applications, 24,
  Oxford University Press, 2003.

\end{thebibliography}
\end{document}